\documentclass[a4paper]{amsart}
\usepackage[english]{babel}
\usepackage{amsmath,amsthm,amssymb,amsfonts}
\usepackage{bbm}
\usepackage{enumitem}

\usepackage{hyperref}

\newtheorem{thm}{Theorem}[section]
\newtheorem{lem}[thm]{Lemma}
\newtheorem{prp}[thm]{Proposition}
\newtheorem{cor}[thm]{Corollary}

\newcommand{\C}{\mathbb{C}}
\newcommand{\R}{\mathbb{R}}
\newcommand{\N}{\mathbb{N}}
\newcommand{\Z}{\mathbb{Z}}

\newcommand{\Quat}{\mathbb{H}}
\newcommand{\Qi}{\mathbbm{i}}
\newcommand{\Qj}{\mathbbm{j}}
\newcommand{\Qk}{\mathbbm{k}}
\newcommand{\Qu}{\mathbbm{u}}

\newcommand{\IQ}{I_\Quat}

\newcommand{\QHarm}[2]{\mathcal{H}_{#1,#2}}
\newcommand{\QPoly}[2]{\mathcal{P}_{#1,#2}}

\newcommand{\Poly}{\mathcal{P}}
\newcommand{\hPoly}[1]{\mathcal{P}_{#1}}
\newcommand{\Harm}{\mathcal{H}}
\newcommand{\hHarm}[1]{\mathcal{H}_{#1}}

\newcommand{\CPoly}[2]{\mathcal{Q}_{#1,#2}}
\newcommand{\CHarm}[2]{\mathcal{Y}_{#1,#2}}

\newcommand{\pU}{U}   
\newcommand{\pG}{J}   
\newcommand{\pP}{P}   

\newcommand{\osum}{\mathop{\widehat\bigoplus}}

\newcommand{\Sphere}{\mathbb{S}}          
\newcommand{\weight}{\varpi}              

\newcommand{\tc}{:}

\DeclareMathOperator{\id}{id}
\DeclareMathOperator*{\esssup}{ess\,sup}  

\newcommand{\opL}{\mathfrak{L}}           
\newcommand{\opLB}{{\Delta_\Sphere}}      
\newcommand{\nablaLB}{\nabla_\Sphere}     
\newcommand{\opMS}{\Gamma}                
\newcommand{\opDg}{\Theta}                
\newcommand{\opGL}{\Delta}                
\newcommand{\nablaGL}{\nabla}             
\newcommand{\nablaL}{\nabla_H}            

\newcommand{\ladderU}{T_\uparrow}         
\newcommand{\ladderD}{T_\downarrow}
\newcommand{\ladderR}{T_\rightarrow}

\newcommand{\group}[1]{\mathrm{#1}}       

\newcommand{\lie}{\mathfrak}

\newcommand{\RSym}{\group{O}(4n)}
\newcommand{\Sp}{\group{Sp}}
\newcommand{\QSymR}{\Sp(n)}
\newcommand{\QSymL}{\overline{\Sp(1)}}
\newcommand{\QSym}{\QSymL \cdot \QSymR}

\newcommand{\dist}{\varrho}                        

\newcommand{\VV}{\mathcal{V}}
\newcommand{\WW}{\mathcal{W}}

\newcommand{\sloc}{\mathrm{sloc}}

\newcommand{\inner}[2]{\langle #1 | #2 \rangle}   

\newcommand{\smc}{\mathrm{c}}
\newcommand{\sms}{\mathrm{s}}

\numberwithin{equation}{section}

\begin{document}
\title[A sharp multiplier theorem on quaternionic spheres]{Quaternionic spherical harmonics and a sharp multiplier theorem on quaternionic spheres}

\author[J. Ahrens]{Julian Ahrens}
\address{Deutsches Forschungszentrum \\ f\"ur K\"unstliche Intelligenz \\ Trippstadter Str.\ 122 \\ D-67663 Kaiserslautern \\ Germany}
\email{julian.ahrens@dfki.de}
\author[M. G. Cowling]{Michael G. Cowling}
\address{School of Mathematics \\ University of New South Wales \\ UNSW Sydney NSW 2052 \\ Australia}
\email{m.cowling@unsw.edu.au}
\author[A. Martini]{Alessio Martini}
\address{School of Mathematics\\ University of Birmingham\\Edgbaston\\Birmingham B15 2TT \\ United Kingdom}
\email{a.martini@bham.ac.uk}
\author[D. M\"uller]{Detlef M\"uller}
\address{Mathematisches Seminar \\ C.-A.-Universit\"at zu Kiel \\ Ludewig-Meyn-Str.\ 4 \\ D-24118 Kiel \\ Germany}
\email{mueller@math.uni-kiel.de}



\thanks{Cowling was supported by the Australian Research Council, through grant DP140100531.
Martini is a member of the Gruppo Nazionale per l'Analisi Matema\-tica, la Probabilit\`a e le loro Applicazioni (GNAMPA) of the Istituto Nazionale di Alta Matematica (INdAM).
M\"uller was supported by the Deutsche Forschungsgemeinschaft, through grant MU 761/11-1.}

\begin{abstract}
A sharp $L^p$ spectral multiplier theorem of Mihlin--H\"ormander type is proved for a distinguished sub-Laplacian on quaternionic spheres.
This is the first such result on compact sub-Riemannian manifolds where the horizontal space has corank greater than one.
The proof hinges on the analysis of the quaternionic spherical harmonic decomposition, of which we present an elementary derivation.
\end{abstract}

\maketitle

\section{Introduction}

Let $\opGL$ be the Laplacian in Euclidean space.
The investigation of the relation between the $L^p$-boundedness of functions $F(\opGL)$ of the operator and the size and smoothness of the ``spectral multiplier'' $F$ is a classical but still very active area of research of harmonic analysis, with important open problems such as the Bochner--Riesz conjecture.
Analogous problems have been investigated in non-Euclidean settings, and a number of optimal results have been proved when the Laplacian is replaced by a more general self-adjoint elliptic operator on a manifold, such as the Laplace--Beltrami operator on a compact Riemannian manifold.
However, weakening the ellipticity assumption on the operator by passing to sub-elliptic operators, whose underlying geometry is considerably more complex than in  the Riemannian case, leads to substantial new challenges, and very little is known about sharp results in this context.
This work is part of a programme aiming at shedding some light on this problem.
Here we consider a sub-elliptic operator in a setting that presents several new difficulties.
Despite these, we are able to prove a sharp spectral multiplier theorem via a delicate analysis of spherical harmonics on quaternionic spheres.

Let $\Quat$ be the skew field of quaternions. Recall that $\Quat$ is a $4$-dimensional associative unital algebra over $\R$. Each element $x \in \Quat$ may be uniquely written as
\begin{equation}\label{eq:quaternion}
x = a+b\Qi+c\Qj+d\Qk,
\end{equation}
where $a,b,c,d \in \R$ and the quaternionic imaginary units $\Qi,\Qj,\Qk$ satisfy the relations
\begin{equation}\label{eq:quaternioncommutation}
\Qi^2 = \Qj^2 = \Qk^2 = \Qi \Qj \Qk = -1.
\end{equation}
For $x \in \Quat$ as in \eqref{eq:quaternion}, we denote by $\Re x$, $\Im x$, $\overline{x}$ and $|x|$ the real part, the imaginary part, the conjugate and the modulus of $x$, given by
\[
\Re x = a, \qquad \Im x = b\Qi+c\Qj+d\Qk, \qquad \overline{x} = \Re x-\Im x, \qquad |x| = \sqrt{x \overline x}.
\]

Let $n \in \N$ be greater than $1$. We consider $\Quat^n$ as a left $\Quat$-module.
Define the quaternionic inner product $\langle \cdot,\cdot \rangle : \Quat^n \times \Quat^n \to \Quat$ by
\[
\langle x,y \rangle = \sum_{j=1}^n x_j \, \overline{y_j}
\]
for all $x=(x_1,\dots,x_n)$ and $y=(y_1,\dots,y_n)$ in $\Quat^n$. The real part $\Re \langle \cdot,\cdot \rangle$ is the usual $\R$-bilinear inner product on $\Quat^n$, corresponding to the identification of $\Quat^n$ with $\R^{4n}$.

Let $\Sphere$ be the unit sphere in $\Quat^n$:
\[
\Sphere = \{ x \in \Quat^n \tc \langle x, x \rangle = 1\}.
\]
Then $\Sphere$ is a smooth real hypersurface in $\Quat^n$, that is, $\dim \Sphere = 4n-1$.
As usual, the tangent space $T_x \Sphere$ at each point $x \in \Sphere$ may be identified with a $1$-codimensional $\R$-linear subspace of $\Quat^n$, given by
\[
T_x \Sphere = \{ y \in \Quat^n \tc \Re \langle x,y \rangle = 0 \}.
\]
The restriction of the inner product $\Re \langle \cdot,\cdot \rangle$ to each tangent space determines a Riemannian metric on $\Sphere$. Unless otherwise specified, integration on $\Sphere$ is considered with respect to the rotation-invariant probability measure $\sigma$ on $\Sphere$.

Let $H \Sphere$ be the tangent distribution on $\Sphere$ of corank $3$ defined by
\[
H_x \Sphere = \{ y \in \Quat^n \tc \langle x,y \rangle =0\}
\]
for all $x \in \Sphere$. It may be shown that $H\Sphere$ is bracket-generating (see \cite{biquard_quaternionic_1999,astengo_cayley_2004,baudoin_subelliptic_2014}). So, together with the Riemannian metric, it determines a sub-Riemannian structure on $\Sphere$, whose horizontal distribution is $H\Sphere$. We denote the corresponding intrinsic sub-Laplacian (see \cite{montgomery_tour_2002,agrachev_intrinsic_2009}) by $\opL$.

A more explicit description of the horizontal distribution $H\Sphere$ and the sub-Laplacian $\opL$ may be given.
It is easily checked that the vector fields
\begin{equation}\label{eq:unitfields}
T_\Qi : x \mapsto -\Qi x,\qquad T_\Qj : x \mapsto -\Qj x, \qquad T_\Qk : x \mapsto -\Qk x
\end{equation}
are tangent to the sphere $\Sphere$, and that we have the orthogonal decomposition
\[
T_x \Sphere = H_x \Sphere \oplus \R T_\Qi|_x \oplus \R T_\Qj|_x \oplus \R T_\Qk|_x
\]
for all $x \in \Sphere$. Indeed $T_\Qi|_x,T_\Qj|_x,T_\Qk|_x$ form an orthonormal basis of the orthogonal complement of $H_x \Sphere$ in $T_x \Sphere$ for all $x \in \Sphere$. Correspondingly, for all real-valued smooth functions $f$ on the sphere $\Sphere$, the Riemannian gradient $\nablaLB f$ may be written as
\[
\nablaLB f = \nablaL f + (T_\Qi f) T_\Qi  + (T_\Qj f) T_\Qj + (T_\Qk f) T_\Qk,
\]
where $\nablaL$ denotes the horizontal gradient associated with $H\Sphere$ (that is, the projection onto $H\Sphere$ of the Riemannian gradient) and the vector fields $T_\Qi,T_\Qj,T_\Qk$ are identified with first-order differential operators as usual. In particular,
\[
\Re \langle \nablaLB f, \nablaLB g\rangle = \Re \langle \nablaL f, \nablaL g\rangle + \sum_{\Qu \in \{\Qi,\Qj,\Qk\}} (T_\Qu f) (T_\mathbbm{u} g)
\]
for all real-valued smooth functions $f,g$ on $\Sphere$. Taking integrals over $\Sphere$ and then integrating by parts finally gives that
\[
\opLB = \opL + \opMS,
\]
where $\opLB$ is the Laplace--Beltrami operator on $\Sphere$ and $\opMS = -(T_\Qi^2+T_\Qj^2+T_\Qk^2)$.

The sub-Laplacian $\opL$ is a nonnegative essentially self-adjoint hypoelliptic operator on $L^2(\Sphere)$.
Hence a functional calculus for $\opL$ may be defined via the spectral theorem and, for all bounded Borel functions $F : \R \to \C$, the operator $F(\opL)$ is bounded on $L^2(\Sphere)$. Here we are interested in the problem of finding sufficient conditions on the function $F$ so that the operator $F(\opL)$, initially defined on $L^2(\Sphere)$, extends to a bounded operator on $L^p(\Sphere)$ for some $p \neq 2$.

For all $s \in [0,\infty)$, let $L^2_s(\R)$ denote the $L^2$ Sobolev space on $\R$ of (fractional) order $s$. We also define a local scale-invariant Sobolev norm as follows: for a Borel function $F : \R \to \C$, set
\[
\|F\|_{L^2_{s,\sloc}} = \sup_{t \geq 0} \| F(t \cdot) \, \chi \|_{L^2_s(\R)},
\]
for any fixed nonzero cutoff function $\chi \in C^\infty_c((0,\infty))$. Note that different choices of $\chi$ give rise to equivalent norms. Note moreover that $\|F\|_{L^2_{s,\sloc}} \gtrsim |F(0)|$, since the value $t=0$ is included in the supremum above.

By Sobolev's embedding theorem, if $\|F\|_{L^2_{s,\sloc}} < \infty$ for some $s>1/2$, then $F$ agrees almost everywhere with a continuous function on $(0,\infty)$. It is this continuous version of $F$ that features in the first of our main results, which is an $L^p$ spectral multiplier theorem of Mihlin--H\"ormander type for the sub-Laplacian $\opL$.

\begin{thm}\label{thm:mainMH}
If $F : \R \to \C$ is a Borel function which is continuous on $(0,\infty)$, and $\|F\|_{L^2_{s,\sloc}} < \infty$ for some $s > (4n-1)/2$, then the operator $F(\opL)$ is of weak type $(1,1)$ and bounded on $L^p(\Sphere)$ for all $p \in (1,\infty)$, and moreover
\[
\| F(\opL) \|_{L^1 \to L^{1,\infty}} \leq C_s \, \|F\|_{L^2_{s,\sloc}}, \qquad \| F(\opL) \|_{L^p \to L^{p}} \leq C_{s,p} \, \|F\|_{L^2_{s,\sloc}}.
\]
\end{thm}

Note that $L^1$-boundedness of $F(\opL)$ in general does not hold under the assumptions of the previous theorem. However we can recover $L^1$-boundedness in the case $F$ is compactly supported.

\begin{thm}\label{thm:maincpt}
If $F : \R \to \C$ is a continuous function supported in $[-1,1]$ and $\|F\|_{L^2_s} < \infty$ for some $s > (4n-1)/2$, then the operator $F(t\opL)$ is bounded on $L^p(\Sphere)$ for all $t \in (0,\infty)$ and $p \in [1,\infty]$, and moreover
\[
\sup_{t > 0} \|F(t\opL)\|_{L^p \to L^p} \leq C_{s,p} \, \|F\|_{L^2_s}.
\]
\end{thm}

Consequently, via complex interpolation, we immediately obtain an $L^p$ boundedness result for the Bochner--Riesz means associated to the sub-Laplacian $\opL$.

\begin{cor}
For all $p \in [1,\infty]$ and $\alpha > (4n-2)|1/2-1/p|$, the Bochner--Riesz means $(1-t\opL)^\alpha_+$ are bounded on $L^p(\Sphere)$ uniformly in $t \in [0,\infty)$.
\end{cor}

One reason of interest of the above results is that the critical index $(4n-1)/2$ in the statements of Theorems \ref{thm:mainMH} and \ref{thm:maincpt} is sharp, in the sense that it cannot be replaced by any smaller number.

Indeed it would be relatively straightforward to derive from the general results of \cite{hebisch_functional_1995,cowling_spectral_2001,duong_plancherel-type_2002} a weaker version of Theorems \ref{thm:mainMH} and \ref{thm:maincpt}, where the $L^2$ Sobolev norm is replaced by an $L^\infty$ Sobolev norm and the critical index $(4n-1)/2$ is replaced by $(4n+2)/2$. Here the value $4n+2$ is the ``local dimension'' associated with the sub-Riemannian structure on $\Sphere$; more precisely, if $\dist$ is the sub-Riemannian (or Carnot--Carath\'eodory) distance function on $\Sphere$, then
\[
\sigma(B(x,r)) \simeq \min \{1, r^{4n+2}\}
\]
for all $x \in \Sphere$ and $r \in (0,\infty)$, where $B(x,r) = \{y \in \Sphere \tc \dist(x,y) < r \}$ denotes the sub-Riemannian ball of centre $x$ and radius $r$ (see Proposition \ref{prp:est}\ref{en:est_invweight} below). The fact that the local dimension associated with $\dist$ is strictly larger than the topological dimension $4n-1$ of the manifold $\Sphere$ is connected with the lack of ellipticity of the sub-Laplacian $\opL$ \cite{fefferman_subelliptic_1983}.

Note that the statements of Theorems \ref{thm:mainMH} and \ref{thm:maincpt} remain true (and sharp) when the sub-Laplacian $\opL$ is replaced by the Laplace--Beltrami operator $\opLB$ on $\Sphere$. However the results for $\opLB$ are particular instances of more general results for elliptic operators on compact manifolds \cite{seeger_boundedness_1989}.
In contrast to the elliptic case,
the problem of obtaining analogous sharp results for sub-Laplacians, with a similar degree of generality to \cite{seeger_boundedness_1989}, appears to be still wide open.

Sharp multiplier theorems for sub-Laplacians are known in a few particular cases. Among these, the results of \cite{cowling_spectral_2011} for a distinguished sub-Laplacian on the unit sphere in $\C^k$ are a natural predecessor of ours. When $k = 2n$, the sphere considered in \cite{cowling_spectral_2011} coincides with our $\Sphere$ as a manifold; however here we study a different sub-Laplacian, associated with a different sub-Riemannian structure. Indeed in \cite{cowling_spectral_2011} the horizontal distribution has corank $1$, rather than $3$. Actually, our result is the first that we are aware of that applies to a sub-Laplacian on a compact sub-Riemannian manifold of corank greater than $1$.

Another case where sharp multiplier theorems are known is that of homogeneous sub-Laplacians on certain classes of
2-step
stratified groups \cite{hebisch_multiplier_1993,mller_spectral_1994,martini_heisenbergreiter,martini_further,martini_necessary}.
Homogeneous sub-Laplacians on stratified groups are of particular relevance, in that they serve as ``local models'' for more general sub-Laplacians on sub-Riemannian manifolds (in much the same way as the Euclidean Laplace operator is a local model for second-order elliptic operators on manifolds). Indeed the sub-Laplacian on the sphere in $\C^k$ studied in \cite{cowling_spectral_2011} is locally modelled on a homogeneous sub-Laplacian on the $(2k-1)$-dimensional Heisenberg group $H_{k-1}$, while the sub-Laplacian $\opL$ on $\Sphere$ considered here corresponds to a homogeneous sub-Laplacian on the $(4n-1)$-dimensional quaternionic Heisenberg group $\Quat H_{n-1}$ (see \cite[Theorem 2.7]{astengo_cayley_2004}). In particular,  by means of a transplantation argument \cite[Section 5]{martini_crsphere}, the sharpness of Theorems \ref{thm:mainMH} and \ref{thm:maincpt} may be derived from the results of \cite{martini_necessary}.

The basic approach to the proof of Theorems \ref{thm:mainMH} and \ref{thm:maincpt} will follow the scheme of \cite{cowling_spectral_2011}. Namely, since $\opL$ has finite propagation speed with respect to $\dist$ \cite{melrose_propagation_1986,cowling_subfinsler_2013},
Theorems \ref{thm:mainMH} and \ref{thm:maincpt} may be reduced, by means of general results proved in \cite{cowling_spectral_2001} (see also \cite[Theorems 2.1 and 2.2]{casarino_spectral_sphere}), to a certain set of estimates, that are listed in Proposition \ref{prp:est} below.

To state the required estimates, it is convenient to introduce some notation. For all Borel functions $F : \R \to \C$ supported in $[0,1]$ and all $N \in \N \setminus \{0\}$, define
\[
\|F\|_{N,2} = \left( \frac{1}{N} \sum_{k=1}^N \sup_{\lambda \in [(k-1)/N,k/N]} |F(\lambda)|^2 \right)^{1/2}
\]
(see \cite[eq.\ (2.5)]{cowling_spectral_2001}). Moreover, for all bounded operators $T$ on $L^2(\Sphere)$, denote by $K_T$ the distributional Schwartz kernel of $T$, formally viewed as an integral kernel, so that
\[
\langle T f, g \rangle_{L^2(\Sphere)} = \int_{\Sphere} \int_{\Sphere} K_{T}(x,y) \, f(y) \, \overline{g(x)} \,d\sigma(x) \,d\sigma(y)
\]
for all $f,g \in C^\infty(\Sphere)$ when the kernel is indeed a function; in general the double integral in the right-hand side is intended in the sense of distributions.

\begin{prp}\label{prp:est}
Let $\weight : \Sphere \times \Sphere \to \R$ be defined by
\begin{equation}\label{eq:weight}
\weight(x,y) = \sqrt{1-|\langle x,y \rangle|^2}
\end{equation}
for all $x,y \in \Sphere$. Then the following estimates hold.
\begin{enumerate}[label=(\roman*)]
\item\label{en:est_invweight} For all $\alpha \in [0,3)$, all $x \in \Sphere$ and all $r \in (0,\infty)$,
\[
\int_{B(x,r)} \weight(x,y)^{-\alpha} \,d\sigma(y) \leq C_\alpha \min\{r^{4n+2-\alpha}, 1\}.
\]
\item\label{en:est_plancherel} For all $\alpha \in [0,3)$, all $N \in \N \setminus \{0\}$, and all bounded Borel functions $F : \R \to \C$ vanishing outside $[0,N)$,
\[
\esssup_{y \in \Sphere} \int_\Sphere |K_{F(\sqrt{\opL})}(x,y)|^2 \, \weight(x,y)^\alpha \,d\sigma(x) \leq C_\alpha \, N^{4n+2-\alpha} \|F(N \cdot)\|_{N,2}^2.
\]
\item\label{en:est_ondiagonal} For all sufficiently large $\ell \in \N$, and for all $x \in \Sphere$ and $r \in (0,\infty)$,
\[
\sigma(B(x,r))^{1/2} \| (1+r^2\opL)^{-\ell} \|_{L^2 \to L^\infty} \leq C_\ell.
\]
\item\label{en:est_doubling} For all $x \in \Sphere$ and $r \in (0,\infty)$,
\[
\sigma(B(x,2r)) \leq C \sigma(B(x,r)).
\]
\end{enumerate}
\end{prp}

The proof of Proposition \ref{prp:est} may be found in Section \ref{section:multipliers} below. As in \cite{cowling_spectral_2011}, the ``weighted Plancherel-type estimate'' for $\opL$, appearing as part \ref{en:est_plancherel} of Proposition \ref{prp:est}, is the most demanding. Its proof requires a careful analysis of the spectral decomposition of $\opL$, which is developed throughout the paper.

Since $\opL$ and the Laplace--Beltrami operator $\opLB$ commute, the spectral decomposition of $\opL$ may be obtained by refining and recombining the spectral decomposition of $\opLB$. The latter is nothing else than the well-known decomposition into spherical harmonics, that is, the decomposition of $L^2(\Sphere)$ into spaces of homogeneous harmonic polynomials.

A similar observation holds true for complex spheres. In \cite{cowling_spectral_2011} a decomposition into ``complex spherical harmonics'' is considered, that refines the classical (or ``real'') spherical harmonic decomposition on the unit sphere in $\C^k$ and yields the joint spectral decomposition of the Laplace--Beltrami operator and the sub-Laplacian studied there. This complex spherical harmonic decomposition may be easily described in terms of ``complex homogeneity'', once polynomials on $\C^k \cong \R^{2k}$ are represented as polynomials in the ``complex indeterminates'' $z_1,\dots,z_k,\bar z_1,\dots,\bar z_k$.
Namely, the space of homogeneous harmonic polynomials of a given degree $h \in \N$ is decomposed into spaces of $(p,q)$-bihomogeneous polynomials, where $p+q = h$ and $p$ and $q$ denote the degrees with respect to the ``holomorphic indeterminates'' $z_1,\dots,z_k$ and the ``antiholomorphic indeterminates'' $\bar z_1,\dots,\bar z_k$ respectively.

One of the main difficulties in dealing with the quaternionic case is that there does not seem to be a comparably straightforward way of describing the ``quaternionic spherical harmonic decomposition'' (that is, the joint spectral decomposition of $\opL$ and $\opLB$) in terms of homogeneity properties of polynomials, as in the real and complex cases. In addition, despite the fact that $\C$ embeds into $\Quat$ as a subfield and
 $\Quat^n$ may be identified with $\C^{2n}$, the quaternionic spherical harmonic decomposition
is not itself a refinement of the complex spherical harmonic decomposition of $L^2(\Sphere)$ resulting from this identification. In other words, the passage from the complex case to the quaternionic case is substantially different from the passage from the real case to the complex case.

Nevertheless, as it turns out, the complex and quaternionic decompositions are compatible (that is, they admit a common refinement). More is true: the action of the differential operators $T_\Qi,T_\Qj,T_\Qk$ on the space of polynomials on $\Quat^n$ defines a representation of the Lie algebra $\lie{su}(2)$ which ``intertwines'' the two decompositions.
This makes it
 possible to derive a sufficiently detailed description of the quaternionic decomposition from the already known properties of the complex decomposition.
 Below we outline such an approach to the quaternionic spherical harmonic decomposition, which is developed in full detail in the thesis of the first-named author \cite{ahrens_masterarbeit_2016}.

An important role in our analysis is naturally played by invariance properties with respect to certain isometries of the sphere. As a matter of fact, the aforementioned real, complex and quaternionic spherical harmonic decompositions correspond to the decomposition of the space of square-integrable functions on the sphere into irreducible representations of certain groups of isometries of the sphere, and may be subsumed in the analysis of the compact Gelfand pairs $(\group{O}(d),\group{O}(d-1))$, $(\group{U}(k),\group{U}(k-1))$, $(\group{Sp}(n) \times \group{Sp}(1),\group{Sp}(n-1) \times \group{Sp}(1))$ respectively. In particular, several of the properties of the quaternionic spherical harmonic decomposition that we present below may be deduced from more general results about the representation theory of the group $\Sp(n) \times \Sp(1)$ and, as such, may be found elsewhere in the literature (see, for example, \cite{pajas_degenerate_1968,johnson_composition_1977,casarino_joint_eigenfunction}).

In contrast, the approach described here does not rely heavily on representation theory, except perhaps a few elementary and well-known facts about $\lie{su}(2)$. Therefore our presentation is likely to be more readily accessible to a wider audience. Moreover, despite the lack of an evident notion of ``quaternionic homogeneity'', here the quaternionic spherical harmonic decomposition is derived as a byproduct of a more general decomposition of the space of all polynomials on $\Quat^n$ (see Section \ref{section:qsh}).
This differs from the previous approaches of \cite{pajas_degenerate_1968,johnson_composition_1977}, which focus on functions on the sphere and zonal harmonics, and instead is consistent with the homogeneity-based approaches to the real and complex cases. For all these reasons, the elementary approach to quaternionic spherical harmonics presented below may be of independent interest.

\subsection*{Notation}
The letter $C$ and variants such as $C_s$ denote constants, always assumed to be positive, which may vary from one occurrence to the next. The expressions $a \simeq b$ and
$a \lesssim b$ mean that there are constants $C$ and $C'$ such that $Ca \leq b \leq C' a$ and $a \leq Cb$ respectively.
$\N$ denotes the set of natural numbers, including $0$.

\subsection*{Post scriptum}
After this paper was submitted, the results of \cite{martini_necessary2} were discovered; these show that half the topological dimension is a lower bound for the Mihlin--H\"ormander critical index for any sub-Laplacian on a sub-Riemannian manifold of arbitrary step. In particular, the sharpness of Theorems \ref{thm:mainMH} and \ref{thm:maincpt} follows directly from \cite{martini_necessary2}.

\section{Quaternionic spherical harmonics}\label{section:qsh}

The Riemannian structure on $\Sphere$ is clearly rotation-invariant, that is, it is invariant under the natural action of the orthogonal group $\RSym$ on $\Quat^n \cong \R^{4n}$. In fact $\RSym$ may be viewed as the group of $\R$-linear automorphisms of $\Quat^n$ that preserve the real inner product $\Re \langle \cdot,\cdot \rangle$. In particular, the Laplace--Beltrami operator $\opLB$ is $\RSym$-invariant.

The sub-Riemannian structure determined by $H\Sphere$ has a smaller symmetry group. Indeed an arbitrary element of $\RSym$ need not preserve the quaternionic inner product $\langle \cdot, \cdot \rangle$ on $\Quat^n$. However two subgroups of $\RSym$ that preserve the sub-Riemannian structure are easily identified. One is the compact symplectic group $\QSymR$, that is, the group of the $\Quat$-linear elements of $\RSym$. The elements of $\QSymR$ preserve the quaternionic inner product.  The other is the group, that we denote by $\QSymL$, of $\R$-linear transformations of $\Quat^n$ given by left multiplication by unit quaternions, that is, of the form
\[
\Quat^n \ni x \mapsto c x \in \Quat^n,
\]
where $c \in \Quat$ and $|c| = 1$. Not all the elements of $\QSymL$ preserve the quaternionic inner product; however, they do preserve orthogonality with respect to it.

Recall that $\Quat^n$ has been given the structure of a left $\Quat$-module. Since multiplication in $\Quat$ is not commutative, the elements of $\QSymL$ need not be $\Quat$-linear. On the other hand, the elements of $\QSymR$ are $\Quat$-linear and therefore commute with the elements of $\QSymL$, by definition of $\Quat$-linearity.

Note that, if we consider the elements of $\Quat^n$ as row vectors, then $\Quat$-linear endomorphisms of $\Quat^n$ may be represented by $n \times n$ matrices with coefficients in $\Quat$, acting by right multiplication.
In this matrix representation, elements of $\QSymR$ correspond to quaternionic matrices whose rows form an orthonormal $\Quat$-basis of $\Quat^n$. From this it follows easily that $\QSymR$ acts transitively on $\Sphere$, and that $\QSymL$ is isomorphic to (the opposite group of) the group $\Sp(1)$ of $\Quat$-linear isometries of $\Quat^1$.
We remark that, in the sequel, the quaternionic matrix representation of $\QSymR$ will not be used, and, for all $x \in \Quat^n$, we will write $Ax$ to denote the action on $x$ of any element $A$ of $\QSymR$ --- or, more generally, any element $A$ of $\RSym$.

From the above considerations, it is immediate that the subgroup $\QSym$ of $\RSym$ generated by $\QSymL$ and $\QSymR$ is compact.
Moreover, since $\QSym$ preserves orthogonality with respect to the quaternionic inner product on $\Quat^n$, the horizontal distribution $H\Sphere$ is $\QSym$-invariant as well. Consequently both $\opL$ and $\opMS$ are $\QSym$-invariant.

Note that, for all $\Qu \in \{\Qi,\Qj,\Qk\}$,
\begin{equation}\label{eq:QexpTfields}
T_\Qu f(x) = \left.\frac{d}{dt}\right|_{t=0} f(\exp(-t\mathbbm{u}) x),
\end{equation}
where $\exp : \Quat \to \Quat$ denotes the quaternionic exponential map.
Since $\opLB$ is $\QSymL$-invariant, each of the vector fields $T_\Qi$, $T_\Qj$, $T_\Qk$ commutes with $\opLB$, and in particular
$\opLB$ and $\opMS = -(T_\Qi^2+T_\Qj^2+T_\Qk^2)$ commute as well. Hence the analysis of the spectral decomposition of $\opL = \opLB - \opMS$ may be reduced to that of the joint spectral decomposition of $\opLB$ and $\opMS$.

The main result of this section is the description of this joint spectral decomposition, which is stated in the proposition below. The set of indices
\[
\IQ = \{ (h,m) \in \N^2 \tc 2m \leq h \}
\]
will be of use in this description.

\begin{prp}\label{prp:quatharmonics}
There is a Hilbert space orthogonal direct sum decomposition
\begin{equation}\label{eq:quaternionicharmonics}
L^2(\Sphere) = \osum_{(h,m) \in \IQ} \QHarm{h}{m}
\end{equation}
with the following properties.
\begin{enumerate}[label=(\roman*)]
\item\label{en:quatharmonics_dim} $\QHarm{h}{m}$ is $\QSym$-invariant and finite-dimensional, and
\[
\dim \QHarm{h}{m} = \frac{(h-2m+1)^2 (h+2n-1)}{(2n-2)(2n-1)} \binom{h-m+2n-2}{2n-3} \binom{m+2n-3}{2n-3}.
\]

\item\label{en:quatharmonics_eigen} The elements of $\QHarm{h}{m}$ are joint eigenfunctions of $\opLB$ and $\opMS$ of eigenvalues
\[
\lambda_{h,m}^\opLB = h(h+4n-2), \qquad \lambda_{h,m}^\opMS = (h-2m)(h-2m+2)
\]
respectively. In particular, they are eigenfunctions of $\opL$ of eigenvalue
\[
\lambda_{h,m}^\opL = 4m(h-m+1) + 4(n-1)h.
\]
\end{enumerate}
\end{prp}

Note that $T_\Qi,T_\Qj,T_\Qk$ do not commute with one another. Indeed
from the commutation relations \eqref{eq:quaternioncommutation} between $\Qi,\Qj,\Qk$ it follows that
\begin{equation}\label{eq:su2commutation}
[T_\Qi,T_\Qj] = 2 T_\Qk, \qquad [T_\Qj,T_\Qk] = 2 T_\Qi, \qquad [T_\Qk,T_\Qi] = 2 T_\Qj.
\end{equation}
However these relations imply that each of $T_\Qi,T_\Qj,T_\Qk$ commutes with $\opMS$. Hence it is possible to relate the joint spectral decomposition of $\opLB,\opMS$ and that of $\opLB,T_\Qi$ via their common refinement, that is, the joint spectral decomposition of $\opLB,\opMS,T_\Qi$.

As we shall see, in order to describe the joint spectral decomposition of $\opLB,T_\Qi$, it is useful to consider $\Quat^n$ as a complex vector space. Note however that there is more than one complex structure on $\Quat^n$. Indeed every $\Qu \in \{\Qi,\Qj,\Qk\}$ may be viewed as a complex structure on $\Quat^{n}$, because $\Qu^2 = -1$: namely, $\Qu$ induces the structure of a $\C$-vector space on $\Quat^{n}$, where multiplication by the imaginary unit $i \in \C$ corresponds to left multiplication by $\Qu$. This complex structure is orthogonal, in the sense that multiplication by $i$ is a linear isometry of $\Quat^n$. Moreover, with respect to this complex structure, $T_\Qu$ is the same as the vector field $z \mapsto -iz$, that is,
\begin{equation}\label{eq:complex_structure_diff}
T_\Qu f(z) = \left.\frac{d}{dt}\right|_{t=0} f(e^{-it}z).
\end{equation}

The various spectral decompositions that we are interested in will be described in terms of spaces of polynomials on $\Quat^n$. Hence it is convenient to consider extensions of the differential operators on $\Sphere$ introduced so far. Extend the Laplace--Beltrami operator $\opLB$ to a differential operator on $\Quat^n \setminus \{0\}$ as follows: for a smooth function $f$ on $\Quat^n \setminus \{0\}$, $\opLB f(x)$ is defined by applying $\opLB$ to the restriction of $y \mapsto f(|x| \, y)$ to the sphere $\Sphere$ and evaluating the result at $x/|x|$. If $\opGL$ is the usual (nonnegative) Laplace operator on $\Quat^n \cong \R^{4n}$, then the well-known formula for the Laplacian in spherical coordinates gives that
\begin{equation}\label{eq:sphericallaplacian}
\opLB = |\cdot|^2 \opGL + \opDg^2 +(4n-2) \opDg,
\end{equation}
where $\opDg$ is the Euler operator (or degree operator) given by
\[
\opDg f(x) = \left.\frac{d}{dt}\right|_{t=1} f(tx),
\]
and $|\cdot|^2$ is the multiplication operator given by
\[
(|\cdot|^2 f)(x) = |x|^2 f(x).
\]
Moreover the operators $T_\Qi$, $T_\Qj$, $T_\Qk$ are naturally extended to differential operators on $\Quat^n$ (indeed the formulas \eqref{eq:unitfields} define global vector fields on $\Quat^n$), hence the same holds for $\opMS = -\sum_{\Qu \in \{\Qi,\Qj,\Qk\}} T_\Qu^2$.

Let $\Poly$ denote the space of (complex valued) polynomial functions on $\Quat^n \cong \R^{4n}$. Any system $\xi_1,\dots,\xi_{4n}$ of real orthonormal coordinates on $\Quat^n$ may be used as a system of indeterminates for $\Poly$; in other words, any element of $\Poly$ may be uniquely written in the form $\sum_{\alpha \in \N^{4n}} c_\alpha \xi^\alpha$ for some coefficients $c_\alpha \in \C$ (all but finitely many of which are zero), where $\xi^\alpha = \xi_1^{\alpha_1} \cdots \xi_{4n}^{\alpha_{4n}}$ for all multiindices $\alpha \in \N^{4n}$.

For a polynomial $p = \sum_{\alpha \in \N^{4n}} c_\alpha \xi^\alpha \in \Poly$, let $p(\partial)$ denote the constant-coefficient differential operator $\sum_{\alpha \in \N^{4n}} c_{\alpha} \partial^\alpha_\xi$, where $\partial_{\xi}^\alpha = \partial_{\xi_1}^{\alpha_1} \cdots \partial_{\xi_{4n}}^{\alpha_{4n}}$. Note that the operator $p(\partial)$ does not depend on the choice of orthonormal coordinates.

For any given orthogonal complex structure on $\Quat^n$, one may choose real orthonormal coordinates $\xi_1,\dots,\xi_{4n}$ in such a way that the expressions $z_j = \xi_{2j-1} + i \xi_{2j}$ define $\C$-linear functionals on $\Quat^n$ for $j=1,\dots,2n$. If we set $\bar z_j = \xi_{2j-1} - i \xi_{2j}$ for $j=1,\dots,2n$, then every polynomial $p \in \Poly$ may be uniquely written in the form $p = \sum_{\alpha,\beta \in \N^{2n}} c_{\alpha,\beta} z^\alpha \bar z^\beta$ for some coefficients $c_{\alpha,\beta} \in \C$  (all but finitely many of which are zero), where $z^\alpha = z_1^{\alpha_1} \cdots z_{2n}^{\alpha_{2n}}$ and $\bar z^\beta = \bar z_1^{\beta_1} \cdots \bar z_{2n}^{\beta_{2n}}$. The system $z_1,\dots,z_{2n},\bar z_1,\dots,\bar z_{2n}$ will be called a system of complex indeterminates for $\Poly$ compatible with the given complex structure.

Note that each of the operators $\opGL,\opDg,T_\Qi,T_\Qj,T_\Qk,\Gamma,|\cdot|^2$ maps $\Poly$ into $\Poly$.

\begin{lem}[{see \cite[\S IV.2]{stein_introduction_1971}}]\label{lem:polyinnerproduct}
Let $B : \Poly \times \Poly \to \C$ be the sesquilinear form on $\Poly$ defined by
\[
B(p,q) = p(\partial) \overline{q} (0) .
\]
Then $B$ is a (positive definite, hermitian) inner product on $\Poly$. With respect to this inner product, the operators $\opMS, \opDg$ are self-adjoint, the operators $T_\Qi,T_\Qj,T_\Qk$ are skew-adjoint, and the operator $|\cdot|^2$ is the adjoint of $-\opGL$. Moreover each of the operators $T_\Qi,T_\Qj,T_\Qk,\Gamma,|\cdot|^2\Delta$ commutes with $\opDg$.
\end{lem}
\begin{proof}
Let $(\xi_1,\dots,\xi_{4n})$ be real orthonormal coordinates on $\Quat^n$. It is easily seen that
\[
B(\xi^\alpha,\xi^{\alpha'}) = \delta_{\alpha,\alpha'} \, \alpha!
\]
for all $\alpha,\alpha' \in \N^{4n}$. Since $\{\xi^\alpha\}_{\alpha \in \N^{4n}}$ is a basis of $\Poly$, this shows that the sesquilinar form $B$ is hermitian and positive definite, and that $\{\xi^\alpha/\sqrt{\alpha!}\}_{\alpha \in \N^{4n}}$ is an orthonormal basis with respect to $B$.

With respect to this orthonormal basis, the operator $\opDg$ is diagonal, with nonnegative eigenvalues, since
\begin{equation}\label{eq:degree_eigen}
\opDg \xi^\alpha = |\alpha| \, \xi^\alpha,
\end{equation}
where $|\alpha| = \alpha_1 + \dots + \alpha_{4n}$ is the length of the multiindex $\alpha$.
Hence $\Theta$ is self-adjoint, and its eigenspaces in $\Poly$ correspond to the subspaces of homogeneous polynomials.

Note also that
\[
B(|\cdot|^2 p, q) = (|\cdot|^2 p)(\partial) \overline{q}(0) = -\opGL p(\partial)\overline{q} (0) = p(\partial) \overline{(-\opGL q)} (0) = B(p,-\opGL q)
\]
for all $p,q \in \Poly$, whence $|\cdot|^2$ is the adjoint of $-\opGL$.

Take now $\Qu \in \{\Qi,\Qj,\Qk\}$. Introduce complex indeterminates $z_1,\dots,z_{2n}$, $\bar z_1,\dots,\bar z_{2n}$ for $\Poly$ which are compatible with the orthogonal complex structure $\Qu$ on $\Quat^n$. Then
\[
B(z^\alpha \bar z^{\beta}, z^{\alpha'} \bar z^{\beta'}) = \delta_{\alpha,\alpha'} \, \delta_{\beta,\beta'} \, 2^{4n} \, \alpha! \, \beta!,
\]
so $\{z^\alpha \bar z^{\beta} / \sqrt{2^{4n} \alpha! \beta!}\}_{\alpha,\beta \in \N^{2n}}$ is an orthonormal basis of $\Poly$ with respect to $B$. Moreover $T_\Qu$ is diagonal with respect to this basis, with purely imaginary eigenvalues, since, by \eqref{eq:complex_structure_diff},
\begin{equation}\label{eq:complex_eigen}
T_\Qu (z^\alpha \bar z^{\beta}) = i(|\beta|-|\alpha|) z^\alpha \bar z^{\beta}.
\end{equation}
Hence $T_\Qu$ is skew-adjoint with respect to $B$.

Consequently $\opMS = -\sum_{\Qu \in \{\Qi,\Qj,\Qk\}} T_\Qu^2$
is self-adjoint with respect to $B$.

Clearly each the operators $T_\Qi,T_\Qj,T_\Qk,\Gamma,|\cdot|^2\Delta$ preserves homogeneity and degree of polynomials; in other words, each of them preserves the eigenspaces of $\opDg$, and therefore commutes with $\opDg$.
\end{proof}

The previous lemma allows us to recover immediately a few basic results about the classical decomposition in spherical harmonics (see, for example, \cite{stein_introduction_1971,axler_harmonic_2001}).

Let $\Harm$ be the subspace of $\Poly$ of harmonic polynomials, that is, the $p \in \Poly$ such that $\opGL p = 0$.
For all $h \in \N$, let $\hPoly{h}$ denote the subspace of $\Poly$ of polynomials that are homogeneous of degree $h$, and set $\hHarm{h} = \Harm \cap \hPoly{h}$. Clearly
\[
\dim \hPoly{h} = \binom{h+4n-1}{4n-1}.
\]
By \eqref{eq:degree_eigen}, the decomposition
\[
\Poly = \bigoplus_{h \in \N} \hPoly{h}
\]
corresponds to the decomposition of $\Poly$ into eigenspaces of $\opDg$; since $\opDg$ and $|\cdot|^2 \opGL$ commute, we have the corresponding decomposition
\[
\Harm = \bigoplus_{h \in \N} \hHarm{h}.
\]
From the identity \eqref{eq:sphericallaplacian}, it follows immediately that the elements of $\hHarm{h}$ are eigenfunctions of $\opLB$ of eigenvalue $h(h+4n-2)$.

For notational convenience we set $\hPoly{h} = \hHarm{h} = \{0\}$ for $h \in \Z \setminus \N$. Then
\begin{equation}\label{eq:polyharm}
\hPoly{h} = \hHarm{h} \oplus |\cdot|^2 \hPoly{h-2}
\end{equation}
for all $h \in \Z$: indeed $|\cdot|^2 : \hPoly{h-2} \to \hPoly{h}$ is the adjoint of $-\opGL : \hPoly{h} \to \hPoly{h-2}$ with respect to the inner product $B$ defined in Lemma \ref{lem:polyinnerproduct}, so $\hPoly{h}$ may be written as the direct sum of the kernel $\hHarm{h}$ of $-\opGL$ and the range of its adjoint. In particular
\[\begin{split}
\dim \hHarm{h} &= \dim \hPoly{h} - \dim \hPoly{h-2} \\
&= \frac{2h+4n-2}{4n-2} \binom{h+4n-3}{4n-3}.
\end{split}\]

Denote by $\Poly|_\Sphere,\Harm|_\Sphere,\hHarm{h}|_\Sphere$ the sets of restrictions to $\Sphere$ of elements of $\Poly,\Harm,\hHarm{h}$ respectively.
Iteration of \eqref{eq:polyharm} shows that $\Poly|_\Sphere = \Harm|_\Sphere$. In particular, by the Stone--Weierstra\ss\ theorem, $\Harm|_\Sphere$ is dense in $L^2(\Sphere)$. Moreover, by the maximum principle for harmonic functions, each element of $\Harm$ is uniquely determined by its restriction to $\Sphere$. Hence
\begin{equation}\label{eq:realdecomposition_restricted}
\Harm|_\Sphere = \bigoplus_{h \in \N} \hHarm{h}|_\Sphere.
\end{equation}
Since $\opLB$ is self-adjoint on $L^2(\Sphere)$ and the eigenvalues $h(h+4n-2)$ are distinct, the spaces $\hHarm{h}|_\Sphere$ are mutually orthogonal in $L^2(\Sphere)$. Taking the closure in $L^2(\Sphere)$ of \eqref{eq:realdecomposition_restricted} then yields the direct sum decomposition
\begin{equation}\label{eq:realharmonics}
L^2(\Sphere) = \osum_{h \in \N} \hHarm{h}|_\Sphere,
\end{equation}
which is the spectral decomposition of $\opLB$.

Due to the injectivity of the restriction map $\Harm \ni f \mapsto f|_\Sphere \in L^2(\Sphere)$, henceforth we shall generally identify subspaces of $\Harm$ with the corresponding subspaces of $L^2(\Sphere)$ and omit the restriction notation.

Introduce the ladder operators
\[
\ladderR = i T_\Qi, \qquad \ladderU = iT_\Qj - T_\Qk, \qquad \ladderD = iT_\Qj+T_\Qk.
\]
It is worth remarking that $\opMS$, $\ladderR$, $\ladderU$, $\ladderD$ commute with the degree operator $\opDg$. Therefore, as in the case of real spherical harmonics, we will first study the decomposition of $\Poly$ into joint eigenspaces of $\opDg,\opMS,\ladderR$, and then consider its intersection with $\Harm$ (where, by \eqref{eq:sphericallaplacian}, $\opLB$ and $\opDg^2 + (4n-2)\opDg$ coincide).

The decomposition of $\Poly$ into joint eigenspaces of $\opDg$ and $\ladderR$ is then easily obtained and well-known \cite{folland_tangential_1972,nagel_moebius_1976,cowling_spectral_2011}.
Take complex indeterminates $z_1,\dots,z_{2n},\bar z_1,\dots,\bar z_{2n}$ corresponding to the complex structure $\Qi$ on $\Quat^n$.
For all $p,q \in \N$, we may then define the space $\CPoly{p}{q}$ to be the space of bihomogeneous polynomials of degree $p$ in $z_1,\dots,z_{2n}$ and of degree $q$ in $\bar z_1,\dots,\bar z_{2n}$. Clearly
\begin{equation}\label{eq:dimCPoly}
\dim \CPoly{p}{q} = \binom{p+2n-1}{2n-1} \binom{q+2n-1}{2n-1}
\end{equation}
and
\begin{equation}\label{eq:bihomog}
\hPoly{h} = \bigoplus_{\substack{p,q \in \N\\ p+q=h}} \CPoly{p}{q}.
\end{equation}
Moreover, by \eqref{eq:complex_eigen}, every element of $\CPoly{p}{q}$ is an eigenfunction of $\ladderR$ of eigenvalue $\lambda^{\ladderR}_{p,q} = p-q$.

Correspondingly
\[
\hHarm{h} = \bigoplus_{\substack{p,q \in \N\\ p+q=h}} \CHarm{p}{q},
\]
where $\CHarm{p}{q} = \CPoly{p}{q} \cap \Harm$. This decomposition is orthogonal in $L^2(\Sphere)$, because the $\CHarm{p}{q}$ are contained in distinct eigenspaces of the self-adjoint operator $\ladderR$. So from \eqref{eq:realharmonics} it follows that
\[
L^2(\Sphere) = \osum_{p,q \in \N} \CHarm{p}{q}.
\]
Moreover from \eqref{eq:polyharm} we deduce that
\[
\CPoly{p}{q} = \CHarm{p}{q} \oplus |\cdot|^2 \CPoly{p-1}{q-1}
\]
(note that $\lambda^{\ladderR}_{p,q} = \lambda^{\ladderR}_{p-1,q-1}$ and $|\cdot|^2$ and $\ladderR$ commute, because $|\cdot|^2$ is rotation-invariant). In particular
\[\begin{split}
\dim \CHarm{p}{q} &= \dim \CPoly{p}{q} - \dim \CPoly{p-1}{q-1} \\
&= \frac{p+q+2n-1}{2n-1} \binom{p+2n-2}{2n-2} \binom{q+2n-2}{2n-2}.
\end{split}\]

We are now going to use the above information about the joint spectral decomposition of $\opDg$ and $\ladderR$ to give a precise description of the joint spectral decomposition of $\opDg$ and $\opMS$. The link between them is given by the following, well-known elementary results about the representation theory of the Lie algebra $\lie{su}(2)$.

\begin{lem}\label{lem:sl2rep}
Let $\VV$ be a minimal finite-dimensional $\{T_\Qi,T_\Qj,T_\Qk\}$-invari\-ant subspace of $\Poly$.
Then there exists a basis $v_0,\dots,v_{d-1}$ of $\VV$ such that
\begin{equation}\label{eq:sl2repT}
\ladderR v_j = (d-1-2j) v_j,\qquad
\ladderU v_j = 2 j(d-j) v_{j-1}, \qquad
\ladderD v_j = 2 v_{j+1},
\end{equation}
for $j=0,\dots,d-1$, where $v_{-1} = v_d = 0$, and in particular
\begin{equation}\label{eq:sl2repMS}
\opMS|_\VV = (d-1)(d+1) \id_{\VV}.
\end{equation}
\end{lem}
\begin{proof}
Because of the commutation rules \eqref{eq:su2commutation}, a minimal $\{T_\Qi,T_\Qj,T_\Qk\}$-invariant subspace of $\Poly$ is an irreducible representation of the Lie algebra $\lie{su}(2)$, and the above description follows by the (elementary) characterisation of such representations (see, for example, \cite{hall_lie_2003}).
\end{proof}

\begin{lem}\label{lem:sl2rep_dec}
Let $\VV$ be a finite-dimensional $\{T_\Qi,T_\Qj,T_\Qk\}$-invari\-ant subspace of $\Poly$. Then
\begin{equation}\label{eq:MSdiag}
\VV = \bigoplus_{\ell \in \N} E^\VV_\opMS(\ell(\ell+2)) = \bigoplus_{m \in \Z} E^{\VV}_{\ladderR}(m),
\end{equation}
where, for all $\lambda \in \C$, $E^\VV_\opMS(\lambda)$ and $E^\VV_{\ladderR}(\lambda)$ denote the eigenspaces of $\opMS|_{\VV}$ and $\ladderR|_\VV$ of eigenvalue $\lambda$. Moreover, for all $m \in \Z$,
\begin{equation}\label{eq:Teigen}
E^{\VV}_{\ladderR}(m) = \bigoplus_{j \in \N} E^{\VV}_{\ladderR}(m) \cap E^\VV_\opMS((|m|+2j)(|m|+2j+2)),
\end{equation}
and correspondingly̧, for all $\ell \in \N$,
\begin{equation}\label{eq:MSeigen}
E^\VV_\opMS(\ell(\ell+2)) = \bigoplus_{j=0}^\ell E^\VV_\opMS(\ell(\ell+2)) \cap  E^\VV_{\ladderR}(\ell-2j).
\end{equation}
Moreover
\begin{align}
\label{eq:MSup} \ladderU &: E^\VV_\opMS(\ell(\ell+2)) \cap  E^\VV_{\ladderR}(\ell-2j) \to E^\VV_\opMS(\ell(\ell+2)) \cap  E^\VV_{\ladderR}(\ell-2(j-1)),\\
\label{eq:MSdown} \ladderD &: E^\VV_\opMS(\ell(\ell+2)) \cap  E^\VV_{\ladderR}(\ell-2(j-1)) \to E^\VV_\opMS(\ell(\ell+2)) \cap  E^\VV_{\ladderR}(\ell-2j)
\end{align}
are isomorphisms for all $j=1,\dots,\ell$ and $\ell \in \N$. In particular, for all $\ell \in \N$,
\begin{equation}\label{eq:MSdim}
\dim E^\VV_\opMS(\ell(\ell+2)) = (\ell+1) ( \dim  E^\VV_{\ladderR}(\ell) - \dim E^\VV_{\ladderR}(\ell+2)).
\end{equation}
\end{lem}
\begin{proof}
Note that, if $\WW$ is a $\{T_\Qi,T_\Qj,T_\Qk\}$-invariant subspace of $\VV$, then its orthogonal complement $\WW'$ in $\VV$ with respect to the inner product $B$ of Lemma \ref{lem:polyinnerproduct} is $\{T_\Qi,T_\Qj,T_\Qk\}$-invariant as well (because $T_\Qi$, $T_\Qj$ and $T_\Qk$ are skew-adjoint with respect to $B$), and $\VV = \WW \oplus \WW'$. Iteration of this observation shows that $\VV$ may be decomposed as a direct sum of minimal $\{T_\Qi,T_\Qj,T_\Qk\}$-invariant subspaces.

By Lemma \ref{lem:sl2rep}, for each of these subspaces one may choose a basis so that the behaviour of $\ladderR$, $\ladderU$, $\ladderD$ and $\opMS$ is prescribed by \eqref{eq:sl2repT} and \eqref{eq:sl2repMS}. These bases together constitute a basis of $\VV$, and inspection of the behaviour of $\ladderR$, $\ladderU$, $\ladderD$ and $\opMS$ on this basis immediately yields the validity of the decompositions \eqref{eq:MSdiag}, \eqref{eq:Teigen} and \eqref{eq:MSeigen}, as well as the fact that \eqref{eq:MSup} and \eqref{eq:MSdown} are isomorphisms.

In particular, the summands in the right-hand side of \eqref{eq:MSeigen} have all the same dimension, and therefore
\[
\dim E^\VV_\opMS(\ell(\ell+2)) = (\ell+1) \dim ( E^\VV_\opMS(\ell(\ell+2)) \cap  E^\VV_{\ladderR}(\ell) ).
\]
On the other hand, by applying \eqref{eq:Teigen} with $m = \ell$,
\[
\dim E^{\VV}_{\ladderR}(\ell) = \sum_{j \in \N} \dim (E^{\VV}_{\ladderR}(\ell) \cap E^\VV_\opMS((\ell+2j)(\ell+2j+2))),
\]
and also, by applying \eqref{eq:Teigen} with $m= \ell+2$ and using the isomorphism \eqref{eq:MSup} with $\ell$ replaced by $\ell+2j$ and then $j$ replaced by $j+1$, we find that
\[\begin{split}
\dim E^{\VV}_{\ladderR}(\ell+2) &= \sum_{j \in \N} \dim (E^{\VV}_{\ladderR}(\ell+2) \cap E^\VV_\opMS((\ell+2j+2)(\ell+2j+4)))\\
&=\sum_{j \in \N} \dim (E^{\VV}_{\ladderR}(\ell) \cap E^\VV_\opMS((\ell+2j+2)(\ell+2j+4))).
\end{split}\]
Thus, by looking at the difference, we obtain that
\[
\dim  E^\VV_{\ladderR}(\ell) - \dim E^\VV_{\ladderR}(\ell+2) = \dim(E^{\VV}_{\ladderR}(\ell) \cap E^\VV_\opMS(\ell(\ell+2))),
\]
and \eqref{eq:MSdim} follows.
\end{proof}

We may now apply Lemma \ref{lem:sl2rep_dec} to $\hPoly{h}$. From the decomposition \eqref{eq:bihomog}, we know that the only eigenvalues of $\ladderR$ that appear in $\hPoly{h}$ are
\[
-h,-h+2,\dots,h-2,h,
\]
that is, they have the form $\pm (h-2m)$ for some $m \in \{0,\dots,\lfloor h/2 \rfloor\}$.
Hence, by \eqref{eq:sl2repMS} and \eqref{eq:MSdim}, the only eigenvalues of $\opMS$ that may appear in $\hPoly{h}$ are of the form $\lambda^\opMS_{h,m} = (h-2m)(h-2m+2)$ for some $m \in \{0,\dots,\lfloor h/2 \rfloor\}$.

Define now $\QPoly{h}{m}$ to be the subspace of $\hPoly{h}$ made of eigenfunctions of $\opMS$ of eigenvalue $\lambda_{h,m}^\opMS$. Observe that $\lambda^{\ladderR}_{p,q} = p-q=h-2q$ whenever $h=p+q$; so, by \eqref{eq:MSdiag} to \eqref{eq:MSeigen}, putting $\ell = h-2m$ and $j=q-m$, we find that
\[
\hPoly{h} = \bigoplus_{m=0}^{\lfloor h/2 \rfloor} \QPoly{h}{m}, \qquad \QPoly{h}{m} = \bigoplus_{q=m}^{h-m} \QPoly{h}{m} \cap \CPoly{h-q}{q},
\]
and moreover, by \eqref{eq:MSdim} and \eqref{eq:dimCPoly},
\[\begin{split}
\dim \QPoly{h}{m} &= (h-2m+1) (\dim \CPoly{h-m}{m} - \dim \CPoly{h-m+1}{m-1}) \\
&= \frac{(h-2m+1)^2}{2n-1} \binom{h-m+2n-1}{2n-2} \binom{m+2n-2}{2n-2}.
\end{split}\]

Correspondingly, if we define $\QHarm{h}{m} = \QPoly{h}{m} \cap \Harm$, then
\begin{equation}\label{eq:hqharmdec}
\hHarm{h} = \bigoplus_{m=0}^{\lfloor h/2 \rfloor} \QHarm{h}{m}, \qquad \QHarm{h}{m} = \bigoplus_{q=m}^{h-m} \QHarm{h}{m} \cap \CHarm{h-q}{q}.
\end{equation}
Moreover from \eqref{eq:polyharm} it follows that
\[
\QPoly{h}{m} = \QHarm{h}{m} \oplus |\cdot|^2 \QPoly{h-2}{m-1}
\]
(note that $|\cdot|^2$ and $\opMS$ commute and $\lambda^\opMS_{h,m} = \lambda^\opMS_{h-2,m-1}$), and therefore
\[\begin{split}
\dim \QHarm{h}{m} &= \dim \QPoly{h}{m} - \dim \QPoly{h-2}{m-1} \\
&= \frac{(h-2m+1)^2 (h+2n-1)}{(2n-2)(2n-1)} \binom{h-m+2n-2}{2n-3} \binom{m+2n-3}{2n-3}.
\end{split}\]

Orthogonality in $L^2(\Sphere)$ of the decompositions \eqref{eq:hqharmdec} follows because the summands are contained in distinct eigenspaces of the self-adjoint operators $\opMS$ and $\ladderR$. By \eqref{eq:realharmonics} we then conclude that
\[
L^2(\Sphere) = \osum_{(h,m) \in \IQ} \QHarm{h}{m},
\]
so in particular the spaces $\QHarm{h}{m}$ are the joint eigenspaces of $\opLB,\opMS$ in $L^2(\Sphere)$, and therefore they are $\QSym$-invariant. This proves Proposition \ref{prp:quatharmonics}.

\section{Zonal harmonics}

In this section, we obtain explicit formulas for the integral kernels of the orthogonal projection operators associated with the quaternionic spherical harmonic decomposition \eqref{eq:quaternionicharmonics}.
These kernels may be characterised by their invariance properties with respect to subgroups of $\QSym$ and may be thought of the quaternionic analogue of ``zonal spherical harmonics''.

The explicit formulas for these kernels are given in terms of classical orthogonal polynomials. For all $q \in \N$, let $\pU_q$ denote the $q$th Chebyshev polynomial of the second kind, that is,
\[
\pU_q(t) = \sum_{j=0}^{\lfloor q/2 \rfloor} (-1)^j \binom{q-j}{j} (2t)^{q-2j}
\]
\cite[\S10.11, p.\ 185, eq.\ (23)]{erdelyi_higher2_1981}.
Moreover, for all $m,\alpha,\beta \in \N$, define the polynomial $\pG^{(\alpha,\beta)}_m$ by
\[\begin{split}
\pG^{(\alpha,\beta)}_m(t) &= \pP^{(\alpha,\beta)}_m(2t-1) \\
&= \frac{(\beta+m)!}{m! (\alpha+\beta+m)!} \sum_{l=0}^m (-1)^l \binom{m}{l} \frac{(\alpha+\beta+2m-l)!}{(\beta+m-l)!} t^{m-l},
\end{split}\]
where $\pP^{(\alpha,\beta)}_m$ is a Jacobi polynomial (see \cite[\S10.8, p.\ 170, eq.\ (16)]{erdelyi_higher2_1981} and \cite[\S 2.8, p.\ 101]{erdelyi_higher1_1981}).

For later use, we record here some useful identities involving the above polynomials. First of all,
\begin{equation}\label{eq:special_normalization}
\pU_q(1) = q+1, \qquad \pG^{(\alpha,\beta)}_m(1) = \binom{\alpha+m}{m}
\end{equation}
(see \cite[\S 10.11, p.\ 184, eq.\ (2)]{erdelyi_higher2_1981} and \cite[\S 10.8, p.\ 169, eq.\ (3)]{erdelyi_higher2_1981}).
Moreover,
\begin{equation}\label{eq:special_weight}
\begin{split}
t \, \pG^{(\alpha,\beta)}_m(t)
&= \frac{(m+1)(m+\alpha+\beta+1)}{(2m+\alpha+\beta+2)(2m+\alpha+\beta+1)} \pG^{(\alpha,\beta)}_{m+1}(t) \\
&\qquad + \frac{1}{2} \left(1-\frac{(\alpha-\beta)(\alpha+\beta)}{(2m+\alpha+\beta+2)(2m+\alpha+\beta)}\right) \pG^{(\alpha,\beta)}_{m}(t) \\
&\qquad + \frac{(m+\alpha)(m+\beta)}{(2m+\alpha+\beta+1)(2m+\alpha+\beta)} \pG^{(\alpha,\beta)}_{m-1}(t);
\end{split}
\end{equation}
see \cite[\S10.8, p.\ 173, eqs.\ (33) and (36)]{erdelyi_higher2_1981}.

This section is devoted to the proof of the following result, which should be compared to \cite[Theorem 3.1(4)]{johnson_composition_1977}.

\begin{prp}\label{prp:Qzonalharm}
For all $(h,m) \in \IQ$, the following hold.
\begin{enumerate}[label=(\roman*)]
\item\label{en:Qzonalharm1} The integral kernel $Z_{h,m}$ of the orthogonal projection of $L^2(\Sphere)$ onto $\QHarm{h}{m}$ is given by
\begin{multline*}
Z_{h,m}(x,y) = \frac{(h-2m+1)(h+2n-1)}{(2n-2)(2n-1)} \binom{h-m+2n-2}{2n-3} \\
\times |\langle x,y \rangle|^{h-2m} \, \pG^{(2n-3,h-2m+1)}_{m}(|\langle x,y \rangle|^2 ) \,  \pU_{h-2m}\left(\frac{\Re \langle x,y \rangle}{|\langle x,y \rangle|}\right)
\end{multline*}
for all $x,y \in \Sphere$.

\item\label{en:Qzonalharm2} For all $e \in \Sphere$, if $\QSym_e$ is the stabiliser of $e$ in $\QSym$, then the space of $\QSym_e$-invariant elements of $\QHarm{h}{m}$ is $1$-dimensional and spanned by $Z_{h,m}(\cdot,e)$.
\end{enumerate}
\end{prp}

The symbol $Z_{h,m}$ will denote the zero function whenever the indices $h,m$ are out of the range $h,m \in \N$, $m \leq \lfloor h/2 \rfloor$.

In order to prove Proposition \ref{prp:Qzonalharm}, it will be first useful to determine the form of a polynomial which is invariant with respect to a group of isometries.

\begin{lem}\label{lem:poly_invariance}
Let $\inner{\cdot}{\cdot}$ denote the standard inner product on $\R^d$. Let $G$ be a subgroup of the group $\group{O}(d)$ of linear isometries of $\R^d$. Suppose that we have an orthogonal decomposition
\[
\R^d = V \oplus \bigoplus_{j=1}^{m} W_j,
\]
where $G$ fixes $V$, and the spaces $W_j$ are $G$-invariant. Moreover, for all $k =1,\dots,m$, assume that the subgroup $G_k$ of $G$ defined by
\[
G_k = \{ T \in G \tc Tx = x \text{ for all } x \in W_j \text{ and } j=1,\dots,k-1\}
\]
acts transitively on the unit sphere $\{x \in W_k \tc |x| = 1\}$ of $W_k$. Let $b_1,\dots,b_l$ be a basis of $V$, and let $P_j$ be the orthogonal projection onto $W_j$ for $j=1,\dots,m$. Then a polynomial $p$ on $\R^d$ is $G$-invariant if and only if $p$ is of the form
\begin{equation}\label{eq:invariant_polynomial}
p(x) = g(\inner{b_1}{x},\dots,\inner{b_l}{x},|P_1 x|^2,\dots,|P_{m} x|^2)
\end{equation}
for some polynomial $g$ on $\R^{l+m}$.
\end{lem}
\begin{proof}
Clearly a polynomial $p$ of the form \eqref{eq:invariant_polynomial} is $G$-invariant. Indeed, if $P_0$ denotes the orthogonal projection of $\R^d$ onto $V$, then $P_0 T x = T P_0 x = P_0 x$, $\inner{b_j}{x} = \inner{b_j}{P_0 x}$, $P_k T x = T P_k x$, and therefore $\inner{b_j}{Tx} = \inner{b_j}{x}$, $|P_k T x|^2 = |P_k x|^2$ for all $T \in G$, $x \in \R^d$, $j=1,\dots,l$, and $k=1,\dots,m$. So it remains to prove the converse, that is, that every $G$-invariant polynomial on $\R^d$ is of the form \eqref{eq:invariant_polynomial}.

We proceed by induction on $m \in \N$. The case $m=0$ is trivial, since every polynomial $f$ on $\R^d$ may be written in the form $f(x) = g(\inner{b_1}{x},\dots,\inner{b_l}{x})$ (the correspondence $x \mapsto (\inner{b_1}{x},\dots,\inner{b_l}{x})$ is a linear automorphism of $\R^d$). Suppose instead that $m > 0$. Then we apply the inductive hypothesis to the group $G_2$ and the decomposition $\R^d = \tilde V \oplus \bigoplus_{j=1}^{m-1} \tilde W_j$, where $\tilde V= V \oplus W_1$ and $\tilde W_j = W_{j+1}$. In this way, if $\tilde b_1,\dots,\tilde b_r$ is an orthonormal basis of $W_1$, then we obtain that every $G_2$-invariant polynomial $p$ is of the form
\[
p(x) = h(\inner{b_1}{x},\dots,\inner{b_l}{x},\inner{\tilde b_1}{x},\dots,\inner{\tilde b_r}{x},|P_2 x|^2,\dots,|P_m x|^2)
\]
for some polynomial $h$ on $\R^{l+r+m-1}$. Note that every such polynomial $p$ may be uniquely written as a sum:
\[
p(x) = \sum_{\substack{s = (s_2,\dots,s_m) \in \N^{m-1} \\ t = (t_1,\dots,t_l) \in \N^l}} h_{s,t}(\inner{\tilde b_1}{x},\dots,\inner{\tilde b_r}{x}) \prod_{j=1}^l \inner{b_j}{x}^{t_j} \prod_{k=2}^m |P_k x|^{2s_k},
\]
for some polynomials $h_{s,t}$ on $\R^r$ (all but finitely many of which are zero); in particular, if $p$ is $G$-invariant, it follows at once (since the $\inner{b_j}{x}$ and the $|P_k x|^2$ are $G$-invariant) that each of the polynomials $h_{s,t}(\inner{\tilde b_1}{x},\dots,\inner{\tilde b_r}{x})$ is $G$-invariant as well.

Since $\tilde b_1,\dots,\tilde b_r$ is an orthonormal basis of $W_1$, the correspondence
\[
W_1 \ni x \mapsto (\inner{\tilde b_1}{x},\dots,\inner{\tilde b_r}{x}) \in \R^r
\]
is a linear isometric isomorphism. Since $G$ acts transitively on the unit sphere of $W_1$ (and therefore, by linearity, on any sphere centred at the origin), we deduce that, if $h_{s,t}(\inner{\tilde b_1}{x},\dots,\inner{\tilde b_r}{x})$ is $G$-invariant, then $h_{s,t}(y) = h_{s,t}(y')$ for all $y,y' \in \R^r$ with $|y|=|y'|$. In other words, $h_{s,t}$ is $\group{O}(r)$-invariant. So, if we decompose $h_{s,t}$ into its homogeneous components,
\[
h_{s,t} = \sum_{u \in \N} h_{s,t,u},
\]
then each $h_{s,t,u}$ is $O(r)$-invariant as well. On the other hand, by homogeneity, $h_{s,t,u}$ is uniquely determined by its values on the unit sphere of $\R^r$; so, if $h_{s,t,u}$ is constant on the unit sphere of $\R^r$, then $h_{s,t,u}(y) = c_{s,t,u} |y|^u$ for some $c_{s,t,u} \in \C$, and in particular $u$ must be even, unless $c_{s,t,u} = 0$.
From this it follows that $h_{s,t}(y) = q_{s,t}(|y|^2)$ for some polynomial $q_{s,t}$ on $\R$; hence, in this case,
\[
h_{s,t}(\inner{\tilde b_1}{x},\dots,\inner{\tilde b_r}{x}) = q_{s,t}\left(\sum_{j=1}^r \inner{\tilde b_j}{x}^2\right) = q_{s,t}(|P_1 x|^2),
\]
and we are done.
\end{proof}

We now apply the previous lemma to the subgroup of $\RSym$ considered in Proposition \ref{prp:Qzonalharm}.

\begin{prp}\label{prp:Qzonalpoly}
Let $G = \QSym_e$ be the stabilizer of $e \in \Sphere$ in $\QSym$. Then a polynomial $f \in \Poly$ is $G$-invariant if and only if $f$ is of the form
\[
f(x) = \sum_{j,k,l \in \N} c_{j,k,l} \, (\Re \langle x,e \rangle)^j \, |\langle x,e \rangle|^{2k} \, |x|^{2l}
\]
for some coefficients $c_{j,k,l} \in \C$ (all but finitely many of which are zero).
\end{prp}
\begin{proof}
The conclusion would follow if we could apply Lemma \ref{lem:poly_invariance} to the group $G$ and the orthogonal decomposition $\Quat^n = V \oplus W_1 \oplus W_2$, where $V = \R e$ and $W_1 = (\Im \Quat) e$; indeed a polynomial in $\Re \langle x,e \rangle$, $|\langle x,e \rangle|^{2}-(\Re \langle x,e \rangle)^2$, $|x|^2-|\langle x,e \rangle|^{2}$ is the same as a polynomial in $\Re \langle x,e \rangle$, $|\langle x,e \rangle|^{2}$, $|x|^2$.

So we must show that the hypotheses of Lemma \ref{lem:poly_invariance} are satisfied. Recall that an element $A$ of $\QSym$ is an $\R$-linear map on $\Quat^n$ of the form
\begin{equation}\label{eq:qsym_map}
A : x \mapsto c T x,
\end{equation}
where $c$ is a unit quaternion and $T \in \QSymR$. In particular, for all $q \in \Quat$,
\begin{equation}\label{eq:qsym_map_action}
A(qe) = cT(qe) = cq\overline{c} \, cTe = cq\overline{c} \, Ae.
\end{equation}
This immediately shows that, if $A \in G$, that is, $Ae = e$, then $A$ fixes $V = \R e$ as well, and moreover $V \oplus W_1 = \Quat e$ is $A$-invariant. Since $A$ is an isometry, it follows that $W_1 = \Quat e \cap V^\perp$ and $W_2 = (\Quat e)^\perp$ are $A$-invariant too.

Note that the unit sphere in $W_1$ is the set of elements of the form $q e$, where $q$ is an imaginary unit quaternion. Moreover, for all unit quaternions $c$, there exists $T \in \QSymR$ such that $Te = \bar c e$ (indeed $\QSymR$ acts transitively on $\Sphere$), so the map $A$ defined by \eqref{eq:qsym_map} with this $T$ belongs to $G$ and $A(qe) = (c q \bar{c}) e$ by \eqref{eq:qsym_map_action}. Note now that unit quaternions act transitively by conjugation
 on the unit sphere of $\Im \Quat$: indeed, for all $t \in \R$, the matrices of $q \mapsto \exp(t\Qi) q \exp(-t\Qi)$,  $q \mapsto \exp(t\Qj) q \exp(-t\Qj)$, and $q \mapsto \exp(t\Qk) q \exp(-t\Qk)$ with respect to the $\R$-basis $\{\Qi,\Qj,\Qk\}$ of $\Im \Quat$ are
\[
\begin{pmatrix}
1 & 0 & 0 \\
0 & \smc(t) & -\sms(t) \\
0 & \sms(t) & \smc(t)
\end{pmatrix},
\begin{pmatrix}
\smc(t) & 0 & \sms(t) \\
0 & 1 & 0 \\
-\sms(t) & 0 & \smc(t)
\end{pmatrix},
\begin{pmatrix}
\smc(t) & -\sms(t) & 0 \\
\sms(t) & \smc(t) & 0 \\
0 & 0 & 1
\end{pmatrix},
\]
where $\smc(t) = \cos(2t)$ and $\sms(t) = \sin(2t)$.
Hence we conclude that $G$ acts transitively on the unit sphere of $W_1$.

Finally, it is clear that the stabilizer $\QSymR_e$ of $e$ in $\QSymR$ is contained in $G$, $\QSymR_e$ fixes $\Quat e = V \oplus W_1$ and $\QSymR_e \cong \Sp(n-1)$ acts transitively on the unit sphere of $W_2 = (\Quat e)^\perp$.
\end{proof}

We can now determine a more precise expression for the quaternionic zonal harmonics.

\begin{prp}\label{prp:Qhomzonalpoly}
Let $e \in \Sphere$ and $G = \QSym_e$. Let $(h,m)\in \IQ$. Let $f \in \Poly$.
\begin{enumerate}[label=(\roman*)]
\item $f$ is $G$-invariant and in $\QPoly{h}{m}$  if and only if $f$ is of the form
\[
f(x) = |\langle x,e \rangle|^{h-2m} \, \pU_{h-2m}\left(\frac{\Re \langle x,e \rangle}{|\langle x,e \rangle|}\right) \sum_{l=0}^m a_l |x|^{2l} |\langle x,e \rangle|^{2(m-l)} ,
\]
for some coefficients $a_0,\dots,a_m \in \C$.

\item $f$ is $G$-invariant and in $\QHarm{h}{m}$  if and only if $f$ is of the form
\[
f(x) =  a \,|x|^{2m}  |\langle x,e \rangle|^{h-2m} \, \pG^{(2n-3,h-2m+1)}_{m}\left(\frac{|\langle x,e \rangle|^2}{|x|^2}\right)  \pU_{h-2m}\left(\frac{\Re \langle x,e \rangle}{|\langle x,e \rangle|}\right),
\]
for some $a \in \C$.
\end{enumerate}
\end{prp}
\begin{proof}
Let $\nablaGL$ denote the usual (Euclidean) gradient on $\Quat^{n} \cong \R^{4n}$. Then it is not difficult to compute that, for all $j,k,l \in \N$,
\begin{align*}
\nablaGL (\Re \langle x,e \rangle)^j &= j (\Re \langle x,e \rangle)^{j-1} e, & -\opGL (\Re \langle x,e \rangle)^j &= j(j-1) (\Re \langle x,e \rangle)^{j-2},\\
\nablaGL |\langle x,e \rangle|^{2k} &= 2k |\langle x,e \rangle|^{2k-2} \langle x,e\rangle e, & -\opGL |\langle x,e \rangle|^{2k} &= 4 k(k+1) |\langle x,e \rangle|^{2k-2},\\
\nablaGL |x|^{2l} &= 2l |x|^{2l-2} x, & -\opGL |x|^{2l} &= 4l(l+2n-1) |x|^{2l-2},
\end{align*}
(differentiation is always meant with respect to $x$) from which one may derive that
\begin{equation}\label{eq:zonal_GL}
\begin{split}
-\opGL \bigl( (\Re \langle x,e \rangle)^j \, &|\langle x,e \rangle|^{2k} \, |x|^{2l} \bigr) \\
&= j(j-1) (\Re \langle x,e \rangle)^{j-2} \, |\langle x,e \rangle|^{2k} \, |x|^{2l} \\
&\qquad + 4k(k+j+1) (\Re \langle x,e \rangle)^j \, |\langle x,e \rangle|^{2k-2} \, |x|^{2l} \\
&\qquad + 4l(l+2k+j+2n-1) (\Re \langle x,e \rangle)^j \, |\langle x,e \rangle|^{2k} \, |x|^{2l-2}.
\end{split}
\end{equation}

Similarly, for all $j,k,l \in \N$ and all  $\Qu \in \{\Qi,\Qj,\Qk\}$,
\begin{gather*}
T_\Qu |\langle x,e \rangle|^{2k} = 0,
 \qquad T_\Qu |x|^{2l} = 0, \\
T_\Qu (\Re \langle x,e \rangle)^j = j (\Re \langle x,e \rangle)^{j-1} \Re\langle x, \Qu e\rangle, \\
-\opMS (\Re \langle x,e \rangle)^j = j(j-1) (\Re \langle x,e \rangle)^{j-2} |\langle x,e \rangle|^2 - j(j+2) (\Re \langle x,e \rangle)^{j},
\end{gather*}
so
\begin{equation}\label{eq:zonal_MS}
\begin{split}
-\opMS \left( (\Re \langle x,e \rangle)^j \, |\langle x,e \rangle|^{2k} \, |x|^{2l} \right)
&= j(j-1) (\Re \langle x,e \rangle)^{j-2} \, |\langle x,e \rangle|^{2k+2} \, |x|^{2l} \\
&\qquad- j(j+2) (\Re \langle x,e \rangle)^j \, |\langle x,e \rangle|^{2k} \, |x|^{2l}.
\end{split}
\end{equation}

Let now $f \in \hPoly{h}$ be $G$-invariant, where $G = \QSym_e$. By Proposition \ref{prp:Qzonalpoly}, $f$ has the form
\begin{equation}\label{eq:hom_qzon}
f(x) = \sum_{\substack{k,l \in \N \\2k+2l \leq h}} c_{k,l} \, (\Re \langle x,e \rangle)^{h-2k-2l} \, |\langle x,e \rangle|^{2k} \, |x|^{2l}
\end{equation}
for some coefficients $c_{k,l} \in \C$. We now want to obtain conditions on the coefficients that correspond to $f$ being in $\QPoly{h}{m}$ and in $\hHarm{h}$.

Indeed, from \eqref{eq:zonal_MS}, we easily deduce that $f \in \QPoly{h}{m}$ if and only if
\[
(h-2k-2l+2)(h-2k-2l+1) c_{k-1,l} = -4(k+l-m)(h-m-k-l+1) c_{k,l},
\]
for all $k,l \in \N$ with $2k+2l \leq h$, where we stipulate that $c_{-1,l} = 0$. This recurrence relation implies that $c_{k,l}$ may be chosen arbitrarily for $k+l = m$, that $c_{k,l} = 0$ for $k+l < m$, and that the coefficients $c_{k,l}$ for $k+l > m$ are uniquely determined by the previous choices (in particular $c_{k,l} = 0$ for $l > m$). In other words, if we set
\begin{equation}\label{eq:qhom_newcoeff}
c_{m-l,l} = 2^{h-2m} a_l ,
\end{equation}
and define
\[
b^{q}_0 = 2^q, \qquad b^{q}_{j} = -\frac{(q-2j+2)(q-2j+1)}{4j(q-j+1)} b^{q}_{j-1},
\]
that is,
\[
b^{q}_j = (-1)^{-j} 2^{q-2j} \binom{q-j}{j},
\]
then
\[
c_{m-l+j,l} = b^{h-2m}_j a_l,
\]
and we may rewrite $f$ as
\begin{equation}\label{eq:qhomzon}
f(x) =  \sum_{j=0}^{\lfloor (h-2m)/2 \rfloor} b^{h-2m}_j \left(\frac{\Re \langle x,e \rangle}{|\langle x,e \rangle|}\right)^{h-2m-2j} \sum_{l=0}^m  a_l  \, |\langle x,e \rangle|^{h-2l} \, |x|^{2l}.
\end{equation}
Note that, in the above expression, the coefficients $a_l$ are freely chosen.

On the other hand, from \eqref{eq:zonal_GL}, we deduce that the polynomial $f$ given by \eqref{eq:hom_qzon} belongs to $\Harm$ if and only if
\begin{equation}\label{eq:harm_coeff}\begin{split}
(h-2k-2l)(h-2k-2l-1) c_{k,l}  &+ 4 (k+1)(h-k-2l)  c_{k+1,l} \\
&+ 4(l+1)(h-l+2n-2) c_{k,l+1} = 0
\end{split}\end{equation}
for all $k,l \in \N$, where we stipulate that $c_{k,l} = 0$ whenever $k < 0$ or $l<0$ or $2k+2l > h$. If we assume as before that $f \in \QPoly{h}{m}$, so $f$ is given by \eqref{eq:qhomzon}, and specialize the identity \eqref{eq:harm_coeff} to the case where $k = m-l-1$ (so $k+l < m$ and $c_{k,l} = 0$), we obtain that
\[
4(m-l)(h-m+1-l) c_{m-l,l} + 4(l+1)(h-l+2n-2) c_{m-l-1,l+1} = 0,
\]
that is, by \eqref{eq:qhom_newcoeff},
\[
4(m-l)(h-m+1-l) a_l + 4(l+1)(h-l+2n-2) a_{l+1} = 0.
\]
This shows that all the $a_l$ ($0 \leq l \leq m$) in this case are determined by the choice of $a_0$. In other words, if we set
\[
a_0 = \binom{h+2n-2}{m} a,
\]
and we define
\[
A^{n,h,m}_{0}= \binom{h+2n-2}{m}, \qquad A^{n,h,m}_{l+1} = -\frac{(m-l)(h-m-l+1)}{(l+1)(h-l+2n-2)} A^{n,h,m}_l,
\]
that is,
\[
A^{n,h,m}_l = (-1)^l \frac{(h-m+1)!}{(h+2n-2-m)!} \frac{(h-l+2n-2)!}{l! (m-l)! (h-m-l+1)!},
\]
then
\begin{multline*}
f(x) =  a \sum_{j=0}^{\lfloor h/2-m \rfloor} b^{h-2m}_j \left(\frac{\Re \langle x,e \rangle}{|\langle x,e \rangle|}\right)^{h-2m-2j} \\
\times \sum_{l=0}^m  A^{n,h,m}_l  \, \left(\frac{|\langle x,e \rangle|}{|x|}\right)^{2(m-l)} \, |\langle x,e \rangle|^{h-2m} \, |x|^{2m}.
\end{multline*}

The conclusion follows by comparing
$A^{n,h,m}_l$ and $b^{h-2m}_j$ with
the coefficients of $Q_{h-2m}$ and $\pG^{(2n-3,h-2m+1)}_m$.
\end{proof}

In order to complete the proof of Proposition \ref{prp:Qzonalharm}, it remains to determine the correct normalization factors for kernels of orthogonal projections.

\begin{lem}
Let $\VV$ be a finite-dimensional, $\QSym$-invariant subspace of $L^2(\Sphere)$ of continuous functions. Let $K$ be the integral kernel of the orthogonal projection of $L^2(\Sphere)$ onto $\VV$. Then
\begin{equation}\label{eq:proj_element}
K(\cdot,x) \in \VV
\end{equation}
for all $x \in \Sphere$,
\begin{equation}\label{eq:proj_invariance}
K(Tx,Ty) = K(x,y)
\end{equation}
for all $T \in \QSym$ and $x,y \in \Sphere$, and
\begin{equation}\label{eq:proj_norm_diagonal}
K(x,x) = \|K(\cdot,x)\|_2^2 = \dim \VV
\end{equation}
for all $x \in \Sphere$.
\end{lem}
\begin{proof}
If $\{\phi_j\}_j$ is any orthonormal basis of $\VV$, then
\[
K(x,y) = \sum_j \phi_j(x) \overline{\phi_j(y)},
\]
and \eqref{eq:proj_element} follows. Moreover \eqref{eq:proj_invariance} is an immediate consequence of the $\QSym$-invariance of $\VV$.

By \eqref{eq:proj_invariance} and the transitivity of $\QSym$ on $\Sphere$, we obtain that $K(x,x)$ does not depend on $x \in \Sphere$. Integration over $\Sphere$ then gives that
\[
K(x,x) = \int_\Sphere K(y,y) \,d\sigma(y) = \sum_j \|\phi_j\|_2^2 = \dim \VV
\]
for all $x \in \Sphere$. Similarly, for all $y \in \Sphere$,
\[
\int_\Sphere |K(x,y)|^2 \,d\sigma(x) = \sum_j |\phi_j(y)|^2,
\]
but the left-hand side does not depend on $y$, again by \eqref{eq:proj_invariance} and transitivity; therefore, by integration over $\Sphere$,
\[
\int_\Sphere |K(x,y)|^2 \,d\sigma(x) = \sum_j \|\phi_j\|_2^2 = \dim\VV.
\]
This gives \eqref{eq:proj_norm_diagonal}.
\end{proof}

We now prove Proposition \ref{prp:Qzonalharm}. Let $Z_{h,m}$ be the integral kernel of the orthogonal projection of $L^2(\Sphere)$ onto $\QHarm{h}{m}$. For all $e \in \Sphere$, by \eqref{eq:proj_element} and \eqref{eq:proj_invariance},
$Z_{h,m}(\cdot,e)$ is a nonzero $\QSym_e$-invariant element of $\QHarm{h}{m}$. This, in conjunction with Proposition \ref{prp:Qhomzonalpoly}, proves part \ref{en:Qzonalharm2}.

As for part \ref{en:Qzonalharm1}, again from Proposition \ref{prp:Qhomzonalpoly} we obtain that
\[
Z_{h,m}(x,e) =  a \,  |\langle x,e \rangle|^{h-2m} \, \pG^{(2n-3,h-2m+1)}_{m}(|\langle x,e \rangle|^2) \,  \pU_{h-2m}\left(\frac{\Re \langle x,e \rangle}{|\langle x,e \rangle|}\right)
\]
for some $a \in \C$. On the other hand, by \eqref{eq:proj_norm_diagonal},
$Z_{h,m}(e,e) = \dim \QHarm{h}{m}$.
By Proposition \ref{prp:quatharmonics}\ref{en:quatharmonics_dim} and \eqref{eq:special_normalization} we then deduce that
\[
a = \frac{(h-2m+1)(h+2n-1)}{(2n-2)(2n-1)} \binom{h-m+2n-2}{2n-3},
\]
and part \ref{en:Qzonalharm1} follows as well.

\section{Weighted Plancherel estimates}

Thanks to the explicit formulas obtained in Proposition \ref{prp:Qzonalharm}\ref{en:Qzonalharm1}, we may now precisely describe the effect of multiplication by $\weight^4$ on the kernels $Z_{h,m}$, where $\weight$ is the weight defined in \eqref{eq:weight}.

\begin{prp}\label{prp:mult_coefficients}
Let $(h,m) \in \IQ$.
\begin{enumerate}[label=(\roman*)]
\item\label{en:mult_coefficients1} For all $x,y \in \Sphere$,
\[\begin{split}
|\langle x,y\rangle|^2 Z_{h,m}(x,y) &= c_{h,m}^\rightarrow Z_{h,m}(x,y) \\
&\qquad + c_{h,m}^\uparrow Z_{h+2,m+1}(x,y) +   c_{h,m}^\downarrow Z_{h-2,m-1}(x,y),
\end{split}\]
where
\begin{align*}
c_{h,m}^\uparrow &= \frac{(m+1)(h-m+2)}{(h+2n)(h+2n+1)},\\
c_{h,m}^\rightarrow &= \frac{1}{2} \left(1-\frac{(2n-4-h+2m)(h-2m+2n-2)}{(h+2n)(h+2n-2)}\right),\\
c_{h,m}^\downarrow &= \frac{(m+2n-3)(h-m+2n-2)}{(h+2n-3)(h+2n-2)}.
\end{align*}

\item\label{en:mult_coefficients2} For all $x,y \in \Sphere$,
\[\begin{split}
\weight(x,y)^4 Z_{h,m}(x,y) &= \gamma_{h,m}^\rightarrow Z_{h,m}(x,y) \\
&\qquad + \gamma_{h,m}^\uparrow Z_{h+2,m+1}(x,y) + \gamma_{h,m}^\downarrow Z_{h-2,m-1}(x,y)\\
&\qquad + \gamma_{h,m}^{\uparrow\uparrow} Z_{h+4,m+2}(x,y) + \gamma_{h,m}^{\downarrow\downarrow} Z_{h-4,m-2}(x,y),
\end{split}\]
where $\gamma_{h,m}^\rightarrow, \gamma_{h,m}^\uparrow, \gamma_{h,m}^\downarrow, \gamma_{h,m}^{\uparrow\uparrow}, \gamma_{h,m}^{\downarrow\downarrow} \in \R$ and
\[
\gamma_{h,m}^\rightarrow = (1-c_{h,m}^\rightarrow)^2 + c_{h,m}^\uparrow c_{h+2,m+1}^\downarrow + c_{h,m}^\downarrow c_{h-2,m-1}^\uparrow
\]
(here $c_{h-2,m-1}^\uparrow = 0$ when $h<2$ or $m<1$).

\item\label{en:mult_coefficients3} There exists $c_n \in (1,\infty)$ such that
\[
c_n^{-1} \left(\frac{m+1}{h+1}\right)^2 \leq \gamma_{h,m}^\rightarrow \leq c_n \left(\frac{m+1}{h+1}\right)^2.
\]
\end{enumerate}
\end{prp}
\begin{proof}
From \eqref{eq:special_weight} we obtain that
\[\begin{split}
t \, &\pG^{(2n-3,h-2m+1)}_m(t) \\
&= \frac{(m+1)(h-m+2n-1)}{(h+2n)(h+2n-1)} \pG^{(2n-3,h-2m+1)}_{m+1}(t) \\
&\qquad + \frac{1}{2} \left(1-\frac{(2n-4-h+2m)(h-2m+2n-2)}{(h+2n)(h+2n-2)}\right) \pG^{(2n-3,h-2m+1)}_{m}(t) \\
&\qquad + \frac{(m+2n-3)(h-m+1)}{(h+2n-1)(h+2n-2)} \pG^{(2n-3,h-2m+1)}_{m-1}(t).
\end{split}\]
If we write the expression for $|\langle x,y\rangle|^2 Z_{h,m}(x,y)$ given by Proposition \ref{prp:Qzonalharm}\ref{en:Qzonalharm1} and employ the above identity with $t = |\langle x,y\rangle|^2$, simple manipulations give part \ref{en:mult_coefficients1}.

From this we deduce in particular that
\[\begin{split}
\weight(x,y)^2 Z_{h,m}(x,y) &= (1-c_{h,m}^\rightarrow) Z_{h,m}(x,y) \\
&\qquad - c_{h,m}^\uparrow Z_{h+2,m+1}(x,y) - c_{h,m}^\downarrow Z_{h-2,m-1}(x,y),
\end{split}\]
and iteration of this identity gives part \ref{en:mult_coefficients2}.

From the formulas in part \ref{en:mult_coefficients1}, it is easily seen that
\[
c^\uparrow_{h,m} \simeq \frac{m+1}{h+1}, \qquad c^\downarrow_{h,m} \simeq \frac{m+1}{h+1},
\]
(note that $0 \leq 2m \leq h$), hence also
\[
c_{h,m}^\uparrow c_{h+2,m+1}^\downarrow + c_{h,m}^\downarrow c_{h-2,m-1}^\uparrow \simeq \left(\frac{m+1}{h+1}\right)^2.
\]
Moreover
\[\begin{split}
1 - c_{h,m}^\rightarrow &= \frac{1}{2} \left(1+\frac{(2n-4-h+2m)(h-2m+2n-2)}{(h+2n)(h+2n-2)}\right) \\
&= \frac{(h+2n)(2n-4)+2(m+1)(h-m+2)}{(h+2n)(h+2n-2)} \\
&\simeq \frac{m+1}{h+1}.
\end{split}\]
Hence part \ref{en:mult_coefficients3} follows from the formula for $\gamma^\rightarrow_{h,m}$ in part \ref{en:mult_coefficients2}.
\end{proof}

We define a ``kernel polynomial'' to be any finite linear combination of the kernels $Z_{h,m}$; in other words, a kernel polynomial $K$ is an expression of the form
\begin{equation}\label{eq:kernelpolynomial}
K = \sum_{(h,m) \in \IQ} a_{h,m} \, Z_{h,m},
\end{equation}
for some coefficients $a_{h,m} \in \C$, all but finitely many of which are zero. Note that, by \eqref{eq:proj_element} and \eqref{eq:proj_norm_diagonal}, if $K$ is given by \eqref{eq:kernelpolynomial} then
\begin{equation}\label{eq:kernelpolynomial_norm}
\int_\Sphere |K(x,y)|^2 \,d\sigma(x) = \sum_{(h,m) \in \IQ} \dim \QHarm{h}{m} \, |a_{h,m}|^2
\end{equation}
for all $y \in \Sphere$.

Proposition \ref{prp:mult_coefficients} tells us that the operator of multiplication by $\weight^4$ does not act diagonally on the basis $\{Z_{h,m}\}_{(h,m) \in \IQ}$ of kernel polynomials; however only a few parallels to the main diagonal in the matrix of this multiplication operator are nonzero. Hence, as we shall show below, this multiplication operator may be majorized (in $L^2$) by its diagonal component, and for the latter we may clearly describe the fractional powers.

For a kernel polynomial $K$ of the form \eqref{eq:kernelpolynomial} and all $\alpha \in [0,\infty)$, we define $M^\alpha K$ by
\begin{equation}\label{eq:spectral_mult}
M^\alpha K = \sum_{(h,m) \in \IQ} (5 \gamma_{h,m}^\rightarrow)^{\alpha/4} a_{h,m} \, Z_{h,m},
\end{equation}
where the coefficients $\gamma_{h,m}^\rightarrow$ are as in Proposition \ref{prp:mult_coefficients}.

\begin{prp}\label{prp:interpol_ineq}
For all kernel polynomials $K$, all $\alpha \in [0,2]$, and all $y \in \Sphere$,
\begin{equation}\label{eq:interpol_ineq}
\| \weight(\cdot,y)^\alpha K(\cdot,y) \|_2 \leq \| M^\alpha K(\cdot,y) \|_2.
\end{equation}
\end{prp}
\begin{proof}
Let $T$ be the linear operator that, to a sequence $(a_{h,m})_{(h,m) \in \IQ}$ of complex numbers, all but finitely many of which are zero, associates the kernel polynomial $K$ given by \eqref{eq:kernelpolynomial}. Then, by \eqref{eq:kernelpolynomial_norm}, it is easily seen that the the estimate \eqref{eq:interpol_ineq} is equivalent to the statement that, for all $y \in \Sphere$, the linear operator $T$ is bounded from
\[
\ell^2_\alpha(\IQ) = \left\{ (a_{h,m})_{(h,m) \in \IQ} \tc \sum_{(h,m) \in \IQ} \dim \QHarm{h}{m} \, (5\gamma^\rightarrow_{h,m})^{\alpha/2} \, |a_{h,m}|^2 < \infty \right\}
\]
to
\[
L^2_\alpha(\Sphere) = \left\{ f \in L^0(\Sphere) \tc \int_\Sphere \weight(x,y)^{2\alpha} \, |f(x)|^2 \,d\sigma(x) < \infty\right\}
\]
with operator norm at most $1$. By complex interpolation between weighted $L^2$-spaces, it is then sufficient to show \eqref{eq:interpol_ineq} for $\alpha = 0$ and $\alpha = 2$. Indeed the case $\alpha = 0$ is trivial (since equality holds in \eqref{eq:interpol_ineq} in that case), so we are reduced to proving the inequality for $\alpha =2$.

Let $K$ be a kernel polynomial as in \eqref{eq:kernelpolynomial}. Decompose $K = \sum_{j=0}^4 K_j$, where
\[
K_j = \sum_{\substack{(h,m) \in \IQ \\ m \equiv j}} a_{h,m} Z_{h,m}
\]
for $j=0,1,2,3,4$ and $\equiv$ denotes congruence modulo $5$. Then by \eqref{eq:spectral_mult} and Proposition \ref{prp:mult_coefficients}\ref{en:mult_coefficients2} it is easily seen that
\[
5\weight^4 K_j = M^4 K_j + \tilde K_j,
\]
where $\tilde K_j(\cdot,y)$ is orthogonal to $K_j(\cdot,y)$ in $L^2(\Sphere)$. Hence, by the Cauchy--Schwarz inequality and orthogonality,
\[\begin{split}
\| \weight(\cdot,y)^2 K(\cdot,y) \|_2^2
&\leq 5 \sum_{j=0}^4 \| \weight(\cdot,y)^2 K_j(\cdot,y) \|_2^2 \\
&= 5 \sum_{j=0}^4 \langle \weight(\cdot,y)^4 K_j(\cdot,y) , K_j(\cdot,y) \rangle \\
&= \sum_{j=0}^4 \langle M^4 K_j(\cdot,y) , K_j(\cdot,y) \rangle \\
&= \sum_{j=0}^4 \| M^2 K_j(\cdot,y) \|_2^2 = \| M^2 K(\cdot,y) \|_2^2,
\end{split}\]
and we are done.
\end{proof}

\section{The multiplier theorem}\label{section:multipliers}

We are now ready to prove Proposition \ref{prp:est}, from which our main theorems follow.
The next statement collects a few estimates that will be useful in the proof.

\begin{lem}\label{lem:useful_estimates}
The following estimates hold.
\begin{enumerate}[label=(\roman*)]
\item\label{en:dimension_estimate} For all $(h,m) \in \IQ$,
\[
\dim \QHarm{h}{m} \simeq (h+1)^{2n-2} (m+1)^{2n-3} (h-2m+1)^2.
\]
\item\label{en:distance_estimate} The sub-Riemannian distance $\dist$ on $\Sphere$ satisfies
\begin{equation}\label{eq:dist_equivalence}
\dist(x,y) \simeq |1-\langle x,y \rangle|^{1/2}
\end{equation}
for all $x,y \in \Sphere$.
\item\label{en:balls_estimate} The sub-Riemannian balls $B(x,r)$ satisfy
\[
\sigma(B(x,r)) \simeq \min\{r^{4n+2},1\}
\]
for all $x \in \Sphere$ and $r \in (0,\infty)$.
\end{enumerate}
\end{lem}
\begin{proof}
\ref{en:dimension_estimate}. This is an immediate consequence of Proposition \ref{prp:quatharmonics}\ref{en:quatharmonics_dim}.

\ref{en:distance_estimate}. Since both sides of \eqref{eq:dist_equivalence} are $\QSym$-invariant, it is sufficient to consider the case where $y = e = (1,0,\dots,0)$. Hence, if we write $x = (x_1,\dots,x_n) \in \Quat^n$, then we must prove that
\[
\dist(x,e) \simeq |1-x_1|^{1/2}.
\]
Note that the above expressions are both continuous in $x$ and vanish on $\Sphere$ only if $x=e$. Hence, by compactness of $\Sphere$, we only need to prove the equivalence when $x$ is in a small neighbourhood of $e$. In this case,
\[\begin{split}
|1-x_1|^{1/2} &\simeq |1-\Re x_1|^{1/2} + |\Im x_1|^{1/2} \\
&\simeq |1-(\Re x_1)^2|^{1/2} + |\Im x_1|^{1/2} \simeq |x'| + |\Im x_1|^{1/2},
\end{split}\]
where $x' = (x_2,\dots,x_n)$. Moreover, in a neighbourhood of $e$, $(x',\Im x_1)$ is a system of local coordinates of $x$ on the manifold $\Sphere$, which are ``linearly adapted coordinates'' for the $2$-step sub-Riemannian structure on $\Sphere$ \cite[\S 4.2]{bellaiche_tangent_1996} and therefore $\dist(x,e) \simeq |x'| + |\Im x_1|^{1/2}$ as well.

\ref{en:balls_estimate}. As before, by $\QSym$-invariance, we are reduced to proving the result for $x = e = (1,0,\dots,0)$ and, by compactness of $\Sphere$, it is sufficient to consider the case where $r$ is small. If $x\in B(e,r)$ for $r$ sufficiently small, then $(x',\Im x_1)$ is a system of local coordinates for $x$ on the manifold $\Sphere$ and $\dist(x,e) \simeq |x'| + |\Im x_1|^{1/2}$, so
\[
\sigma(B(e,r)) \simeq \int_{(y,u) \in \Quat^{n-1} \times \Im \Quat, \, |y| + |u|^{1/2} \leq r } \,dy \,du \simeq r^{4n+2},
\]
and we are done.
\end{proof}

\begin{proof}[Proof of Proposition \ref{prp:est}]
Recall that in the proofs of statements \ref{en:est_invweight} and \ref{en:est_plancherel} we are assuming that $0 \leq \alpha < 3$.

\ref{en:est_invweight}. By $\QSym$-invariance of $\weight$ and $\sigma$, it is sufficient to prove that
\[
\int_{B(e,r)} \weight(e,x)^{-\alpha} \,d\sigma(x) \leq C_\alpha \min\{r^{4n+2-\alpha}, 1\},
\]
for all $r>0$, where $e = (1,0,\dots,0)$. For all $x=(x_1,\dots,x_n) \in \Quat$, if we write $x' = (x_2,\dots,x_n)$, then $\weight(e,x) = \sqrt{1-|x_1|^2} = |x'|$. From this it follows easily that
\[
\int_\Sphere |x'|^{-\alpha} \,d\sigma(x) \lesssim \int_{y \in \Quat^{n-1}, \, |y| \leq 1} |y|^{-\alpha} \,dy < \infty,
\]
since $\alpha < 4(n-1)$. Hence it is sufficient to consider the case where $r$ is small. In this case, $(x',\Im x_1)$ is a system of local coordinates on the manifold $\Sphere$ for $x \in B(e,r)$, and moreover, by Lemma \ref{lem:useful_estimates}\ref{en:distance_estimate}, $\dist(x,e) \simeq |x'| + |\Im x_1|^{1/2}$; therefore
\[
\int_{B(e,r)} |x'|^{-\alpha} \,d\sigma(x) \lesssim \int_{(y,u) \in \Quat^{n-1} \times \Im \Quat, \, |y|+|u|^{1/2} \leq r} |y|^{-\alpha} \,dy \,du \lesssim r^{4n+2-\alpha}.
\]

\ref{en:est_plancherel}.
By Propositions \ref{prp:quatharmonics} and \ref{prp:Qzonalharm},
\begin{equation}\label{eq:kernel_formula}
K_{F(\sqrt{\opL})} = \sum_{(h,m) \in \IQ} F(\sqrt{\lambda^\opL_{h,m}}) Z_{h,m}
\end{equation}
for all compactly supported bounded Borel functions $F : \R \to \C$, and moreover
\begin{equation}\label{eq:eigen_four}
\lambda^\opL_{h,m}/4 = (h-m+n)(m+n-1)-n(n-1).
\end{equation}
If $F : \R \to \C$ vanishes outside $[0,N)$, then, for all $y \in \Sphere$, by Proposition \ref{prp:interpol_ineq}, formulas \eqref{eq:kernelpolynomial_norm} and \eqref{eq:spectral_mult}, Proposition \ref{prp:mult_coefficients}\ref{en:mult_coefficients3}, and Lemma \ref{lem:useful_estimates}\ref{en:dimension_estimate},
\[\begin{split}
\| \weight^{\alpha/2}(&\cdot,y) \, K_{F(\sqrt{\opL})}(\cdot,y) \|_2^2 \\
&\leq \| M^{\alpha/2} \, K_{F(\sqrt{\opL})}(\cdot,y) \|_2^2 \\
&= \sum_{(h,m) \in \IQ} (5\gamma^\rightarrow_{h,m})^{\alpha/4} \dim \QHarm{h}{m} \, |F(\sqrt{\lambda^\opL_{h,m}})|^2\\
&\simeq \sum_{(h,m) \in \IQ} (m+1)^{2n-3+\alpha/2} (h+1)^{2n-2-\alpha/2} \\
&\qquad\qquad\times (h-2m+1)^2 \, |F(\sqrt{\lambda^\opL_{h,m}})|^2 \\
&\leq \sum_{j=1}^N \sum_{(h,m) \in I_j}
 (m+1)^{2n-3+\alpha/2} (h+1)^{2n-\alpha/2} \, \sup_{[j-1,j)} |F|^2,
\end{split}\]
where $I_j = \{ (h,m) \in \IQ \tc(j-1)^2 \leq \lambda_{h,m}^\opL < j^2\}$.

Therefore, in order to prove \ref{en:est_plancherel}, it is sufficient to show that, for all $j \in \N \setminus \{0\}$,
\[
\sum_{(h,m) \in I_j} (m+1)^{2n-3+\alpha/2} (h+1)^{2n-\alpha/2} \leq C_\alpha \, j^{4n+1-\alpha}.
\]

Note that,
if $(h,m) \in I_j$, then, by \eqref{eq:eigen_four},
\[
(j-1)^2/4 \leq ab - n(n-1) < j^2/4,
\]
where $a = h-m+n$ and $b = m+n-1$; this implies that
\[
m+1 \leq b, \qquad (h+1)/2 \leq a \leq ab \leq 2 n(n-1) j^2 ,
\]
and moreover,
for each choice of $a$, the number of values of $b$ satisfying the above inequality is at most
\[
\frac{j^2/4 + n(n-1)}{a} - \frac{(j-1)^2/4 + n(n-1)}{a} \leq \frac{j}{2a}.
\]

In conclusion,
\[\begin{split}
\sum_{(h,m) \in I_j} &(m+1)^{2n-3+\alpha/2} (h+1)^{2n-\alpha/2} \\
&\leq C_{n,\alpha} \sum_{\substack{ a,b \in \N \\ (j-1)^2/4 \leq ab - n(n-1) < j^2/4}} b^{2n-3+\alpha/2} a^{2n-\alpha/2} \\
&\leq C_{n,\alpha} \, j^{4n-6+\alpha} \sum_{\substack{ a,b \in \N \\ (j-1)^2/4 \leq ab - n(n-1) < j^2/4}} a^{3-\alpha} \\
&\leq C_{n,\alpha} \, j^{4n-5+\alpha} \sum_{a=1}^{2n(n-1) j^2} a^{2-\alpha} \\
&\leq C_{n,\alpha} \, j^{4n+1-\alpha},
\end{split}\]
and we are done; notice that in the very last step the condition $\alpha < 3$ becomes crucial.

\ref{en:est_ondiagonal}. By \eqref{eq:kernel_formula} and \eqref{eq:kernelpolynomial_norm},
\[\begin{split}
\| (1+r^2 \opL)^{-\ell} \|^2_{L^2 \to L^\infty} &= \| (1+r^2 \opL)^{-\ell} \|^2_{L^1 \to L^2} \\
&= \esssup_{y \in \Sphere} \int_\Sphere |K_{(1+r^2 \opL)^{-\ell}}(x,y)|^2 \,d\sigma(x) \\
&= \sum_{(h,m) \in \IQ} \dim\QHarm{h}{m} (1+r^2\lambda_{h,m}^\opL)^{-2\ell}.
\end{split}\]
Hence, as before,
\[\begin{split}
\| (1+r^2 \opL)^{-\ell} \|^2_{L^2 \to L^\infty} &\lesssim \sum_{(h,m) \in \IQ} (h+1)^{2n} (m+1)^{2n-3} (1+r^2\lambda_{h,m}^\opL)^{-2\ell} \\
&\lesssim \sum_{j =1}^\infty \sum_{(h,m) \in I_j} j^{4n-6} (h+1)^3 (1+r^2 (j-1)^2)^{-2\ell} \\
&\lesssim \sum_{j =1}^\infty j^{4n-6} (1+r^2 (j-1)^2)^{-2\ell} \sum_{a = 1}^{2n(n-1) j^2} a^3 \frac{j}{a} \\
&\lesssim \sum_{j =1}^\infty j^{4n+1} (1+r^2 (j-1)^2)^{-2\ell}  \\
&\lesssim \max\{1,r^{-(4n+2)}\}
\end{split}\]
whenever $\ell \geq n+1$, and the conclusion follows from Lemma \ref{lem:useful_estimates}\ref{en:balls_estimate}.

\ref{en:est_doubling}. This is an immediate consequence of Lemma \ref{lem:useful_estimates}\ref{en:balls_estimate}.
\end{proof}


\def\polhk#1{\setbox0=\hbox{#1}{\ooalign{\hidewidth\lower1.5ex\hbox{`}\hidewidth\crcr\unhbox0}}}
\providecommand{\bysame}{\leavevmode\hbox to3em{\hrulefill}\thinspace}

\end{document}